\documentclass[12pt,letterpaper,reqno]{article}
\usepackage[letterpaper,margin=1in,headheight=15pt]{geometry} 
\usepackage{graphicx,bbm,comment}
\usepackage{amsmath, amssymb, amsfonts, stmaryrd, amsthm}
\usepackage[utf8]{inputenc}
\usepackage{epstopdf}
\usepackage{tikz-cd} 
\usepackage{setspace}
\usepackage{xcolor}
\usepackage{tcolorbox}
\usepackage[titletoc,title]{appendix}
\usepackage{enumitem}
\usepackage{ulem}

\usepackage{hyperref}
\definecolor{darkred}{rgb}{0.5,0.15,0.15}
\hypersetup{colorlinks=true,urlcolor=darkred,linkcolor=darkred,citecolor=darkred}
\allowdisplaybreaks

\renewcommand{\ell}{X} 

\newcommand{\be}{\begin{eqnarray}}
\newcommand{\ee}{\end{eqnarray}}
\newcommand{\bea}{\begin{eqnarray}}
\newcommand{\eea}{\end{eqnarray}}

\newcommand{\ben}{\begin{eqnarray}}
\newcommand{\een}{\end{eqnarray}}

\theoremstyle{definition}
\newtheorem{thm}{Theorem}
\newtheorem{prop}[thm]{Proposition}
\newtheorem{lem}[thm]{Lemma}
\newtheorem{cor}[thm]{Corollary}
\newtheorem{dfn}[thm]{Definition}

\newtheorem{rem}[thm]{Remark}
\newtheorem{conj}[thm]{Conjecture}

\numberwithin{equation}{section}

\title{A hyperk\"ahler geometry associated to the BPS structure of the resolved conifold}

\author{Murad Alim, Arpan Saha and Iv\'an Tulli}
\date{}

\begin{document}

\maketitle

\begin{abstract}
We associate to the resolved conifold an affine special K\"{a}hler (ASK) manifold of complex dimension $1$, and an instanton corrected hyperk\"{a}hler (HK) manifold of complex dimension $2$. We describe these geometries explicitly, and show that the instanton corrected HK geometry realizes an Ooguri-Vafa-like smoothing of the semi-flat HK metric associated to the ASK geometry. On the other hand, the instanton corrected HK  geometry associated to the resolved conifold can be described in terms of a twistor family of two holomorphic Darboux coordinates. We study a certain conformal limit of the twistor coordinates, and  conjecture a relation to a solution of a Riemann-Hilbert problem previously considered by T. Bridgeland. 
\end{abstract}
\maketitle

\tableofcontents



\section{Introduction}

The study of supersymmetric quantum field theories and string theories has been an extremely rich source of insights for geometry. The spaces of scalar fields which naturally appear in these physical theories lead to mathematical moduli spaces which often carry rich mathematical structures.\\

A particularly interesting class of theories is given by $N=2$ supersymmetric theories in four dimensions which can be obtained from compactifications of ten dimensional type IIA and type IIB string theories on mirror families of Calabi-Yau (CY) threefolds. The scalar fields of the vector multiplets in these theories lead to moduli spaces which are projective (resp.\ affine) special K\"ahler manifolds \cite{Strominger:1990pd,Freed} when the family of mirror CY threefolds is compact (resp.\ non-compact), see \cite{Lledo} for a review. The moduli spaces have different meaning on both sides of mirror symmetry: they correspond to the moduli spaces of complex structures on the B-side and to moduli spaces of complexified K\"ahler forms on the A-side of mirror symmetry. The physical origin of mirror symmetry from 2d superconformal field theories (SCFT) suggests furthermore the identification of the data of a bundle together with a flat connection on both sides of mirror symmetry. This data translates mathematically into a variation of Hodge structures equipped with a flat Gauss-Manin connection which can be used to formulate the geometric constraints of special K\"ahler geometry. The formulation of the moduli space and the flat connection on the A-model side leads to the notion  of quantum cohomology which requires the genus zero Gromov-Witten invariants, see e.~g.~\cite{Cox:2000vi} and references therein.\\

Homological mirror symmetry \cite{HMS} in contrast identifies different triangulated categories attached to a mirror pair of CY threefolds. The bounded derived category of coherent sheaves of one CY gets identified with the derived Fukaya category of the mirror. The moduli spaces in this case are expected to appear as spaces of stability conditions which select special objects in both categories which correspond to the BPS states of the resulting $4d$ theory, see \cite{Aspinwallbook}. For this purpose, the notion of $\pi$-stability was put forward \cite{Douglas:2000ah,Douglas:2000qw,DouglasHMS}. Within $\pi$-stability the stability of objects is governed by the physically motivated notion of a central charge, which is part of the data of the resulting $N=2$ theory.  $\pi$-stability and mirror symmetry predict that the central-charge stability ($\pi$-stability) of the derived category of coherent sheaves of a CY $X$ would require as an input the genus zero Gromov-Witten invariants of $X$. An abstract mathematical notion of $\pi$-stability which is suitable for any triangulated category was put forward by Bridgeland \cite{Bridgeland}, the latter also leads to a notion of a space of stability conditions. \\

This notion of stability is crucial in the study of mathematical invariants such as the Donaldson-Thomas (DT) invariants. A boost in excitement in the study of wall-crossing problems was triggered by new wall-crossing formulas of Kontsevich and Soibelman \cite{KS:2008} as well as Joyce and Song \cite{JS}. The work of Gaiotto, Moore and Neitzke (GMN) \cite{GMN1,GMN2,GMN3} provided moreover a physical embedding of these developments as well as new geometric constructions of hyperk\"ahler manifolds, whose metrics carry the data of the BPS spectra as well as their jumps. See \cite{NewCHK} for an exposition of these ideas. The construction of the hyperk\"ahler metric in this setup relies on the construction of a twistor family of Darboux coordinates for the $\mathcal{O}(2)$-twisted family of holomorphic symplectic forms associated to a hyperk\"{a}hler manifold \cite{HKSUSY}. These coordinates are governed by a system of TBA like equations, whose solutions are in general very hard to describe explicitly. In the simpler case where the BPS spectrum is mutually local, the TBA equations just reduce to integral formulas for the Darboux coordinates, and the resulting hyperk\"{a}hler structure can be described explicitly (even though the coordinates might not). In \cite{Gaiotto:2014bza}, Gaiotto considered a conformal limit of the TBA equations when restricted to a certain Lagrangian submanifold (with respect to one of the holomorphic symplectic structures) of the  hyperk\"ahler manifold associated to the BPS problem. In simple cases, explicit solutions were found involving Gamma functions.\\

In the last few years new exciting insights began to emerge from the study of BPS structures. The new developments are concerned with the analytic and integrable structures behind wall-crossing phenomena. The emerging links provide new connections between DT invariants and GW invariants, going substantially beyond the scope of the previously conjectured relation \cite{MNOP1,MNOP2}. Building on the conformal limit of Gaiotto \cite{Gaiotto:2014bza}, Bridgeland put forward a Riemann-Hilbert problem which transforms the wall-crossing data of Donaldson-Thomas invariants of a given derived category into a tau-function \cite{BridgelandDT}. In simple examples including the resolved conifold \cite{BridgelandCon}, it was shown that an asymptotic expansion of the tau-function leads to the full Gromov-Witten potential. 
\\

The motivation of this current work is two-fold. On one hand, the construction of instanton corrected hyperk\"{a}hler (HK) metrics from \cite{GMN1,NewCHK} suggests that such a metric should improve the behavior of the semi-flat HK metric canonically associated to an affine special K\"{a}hler (ASK) geometry (see \cite[Section 5.3]{NewCHK}). Our first motivation for this work is to try to exemplify this improvement explicitly for an instanton corrected HK geometry associated to the resolved conifold, going beyond the well-understood case of the Ooguri-Vafa space (see \cite[Section 6]{NewCHK} and \cite{LCS}). On the other hand, we would like to try to address an important so far missing connection in the web of ideas outlined above, namely the relation between a conformal limit of the instanton corrected HK geometry studied in \cite{Gaiotto:2014bza}, and Bridgeland's Riemann-Hilbert problem \cite{Bridgeland1,BridgelandCon} associated to the resolved conifold.\\

Our approach to tackle this will be to first search for some natural affine special K\"{a}hler geometry associated to the resolved conifold. Using the BPS spectrum of the resolved conifold, we will then be able to associate an instanton corrected hyperk\"{a}hler geometry, along the lines of \cite{GMN1,NewCHK}. The fact that the BPS spectrum turns out to be mutually local will then allow for a totally explicit expression of the hyperk\"{a}hler structure \cite[Section 3]{QKBPS}. The trickiest part will then be the relation of the conformal limit of the hyperk\"{a}hler geometry, considered in \cite{Gaiotto:2014bza}, with the solution of the Riemann-Hilbert problem considered in \cite{BridgelandCon}. As we will discuss, the conformal limit of  \cite{Gaiotto:2014bza} a priori gives just formal expressions, so we will need to give some prescription for how to make analytical sense of them (at least in the specific case of the HK geometry we associate to the resolved conifold).  We will show that our prescription will satisfy properties related to a solution of the Riemann-Hilbert problem considered in \cite{BridgelandCon}. However, the asymptotic properties that we manage to prove are  weaker than those required to guarantee a match with the unique solution of the Riemann-Hilbert problem, so we will be only be able to conjecture a relation between the conformal limit and the solution of the Riemann-Hilbert problem. \\

We should emphasize that the instanton corrected HK metric we associate considers only a ``truncated" part of the BPS indices.  Indeed, the BPS spectrum of the resolved conifold has the form 

 \begin{equation}\label{BPScon}
\Omega(\gamma) =
  \begin{cases}
    1       & \quad \text{if } \gamma = \pm \beta + n \delta\quad \text{for }\;\; n\in \mathbb{Z},\\
    -2  & \quad \text{if } \gamma= k \delta \quad \text{for} \;\;k\in \mathbb{Z}\setminus \{ 0\},\\
    0 & \quad \text{otherwise}\,.
  \end{cases}
\end{equation}
and the HK geometry we consider will not take into account the part corresponding to the charges $\gamma=k\delta$. The cause of this truncation is the fact that $\delta$ plays the role of a ``flavor" charge in our story. This implies that the corresponding central charge $Z_{\delta}$ does not enter in the description of the underlying ASK geometry (see the discussion in Section \ref{resASK}); and the BPS indices $\Omega(k\delta)$ do not contribute to the instanton corrections of the associated semi-flat geometry (see the second point in Remark \ref{flavorcontribution}). Furthermore, a consequence of this truncation is that the conformal limit associated to our HK geometry only allows us to make relations to the part of the Riemann-Hilbert problem of \cite{BridgelandCon} that involves the Faddeev quantum dilogarithm  \cite{qdilog}.  The reasons for considering $\delta$ as a flavor charge, are as follows:

\begin{itemize}
    \item On one hand, the natural period $Z_{\delta}$ associated to $\delta$ turns out to be constant (see Appendix \ref{appendix:mirrorsymmetry} and Section \ref{rescon}). This suggests that $\delta$ should indeed be treated as a flavor charge, in the language of \cite{GMN1,GMN2, NewCHK}, since it cannot be used as a coordinate for an affine special K\"{a}hler geometry.
    \item If one takes $\delta$ to be a flavor charge, the ASK geometry defined by the periods $Z_{\beta}$ and $Z_{\beta^{\vee}}$ turns out to be  of complex dimension $1$, and the associated intanton corrected HK geometry of complex dimension $2$. On the other hand, the Ooguri-Vafa space is another instanton corrected HK geometry of complex dimension $2$, realizing the smoothing mechanism explained in \cite{NewCHK}. This allows us to explicitly compare the two spaces and show that the instanton corrected HK geometry that we associate to the resolved conifold also realizes the smoothing mechanism via an Ooguri-Vafa-like smoothing. 
    \item On the other hand, if one wants to consider an instanton corrected HK geometry that does take into account the BPS spectrum associated to $\Omega(k\delta)$, one option is to projectivize the periods \eqref{periods}\footnote{This involves setting for  $w\in \mathbb{C}^{\times}$, $Z_{\delta}:=w\varpi^0=w$ and $Z_{\beta}:=w\varpi^1=wt=v$ and writing $Z_{\beta^{\vee}}:=w\varpi^2$ and $Z_{\delta^{\vee}}:=w\varpi^3$ in terms of $w$ and $v$} and consider the (possibly-indefinite)   affine special K\"{a}hler geometry of complex dimension $2$ defined by them. This would give a (possibly-indefinite) HK geometry of complex dimension $4$ taking into account the whole BPS spectrum. However, the details of domains of definition, signature, and what kind of smoothing occurs (if any), is harder to explicitly study. Because of this and the previous two points,  we opted to consider the simpler HK geometry of complex dimension $2$ mentioned before. 
\end{itemize}

On the other hand, a different path to a hyperk\"ahler geometry associated to a BPS problem was put forward by Bridgeland and Strachan in \cite{Bridgeland:2020zjh} and was put in a physical context in \cite{Alexandrov:2021wxu}. Via this perspective one obtains a complex hyperk\"{a}hler geometry, rather than a real one. In the case of the resolved conifold, the complex hyperk\"{a}hler geometry does take into account the full BPS spectrum (\ref{BPScon}), but the physical interpretation of such a complex hyperk\"{a}hler geometry is at the moment unknown. \\

Hence, our approach associates a real HK geometry, but truncates part of the BPS spectrum; while the approach \cite{Bridgeland:2020zjh} considers the full BPS spectrum, but associates a complex HK geometry. The obvious question of the relation between the two HK geometries is certainly an interesting one, but goes beyond the scope of this paper.\\

\textbf{Outline of the paper and statement of main results:}\\
    
  In Section \ref{resASK} we recall what the resolved conifold is, and motivate (together with Appendix \ref{appendix:mirrorsymmetry}) the ingredients of the affine special K\"{a}hler geometry that we will associate to it. We then define the affine special K\"{a}hler (ASK) manifold associated to the resolved conifold, and discuss the semi-flat hyperk\"{a}hler (HK) metric obtained via the rigid c-map \cite{TypeIIgeometry,Freed,ACD}. We will emphasize a description of the ASK and semi-flat HK geometry in terms of certain tuples $(M,\Gamma,Z)$, where $M$ is a complex manifold, $\Gamma\to M$ is a local system of lattices (the ``charge lattice") and $Z$ is a holomorphic section of $\Gamma^*\otimes \mathbb{C}\to M$ (the ``central charge"). This kind of description can be found for example in \cite{GMN1,NewCHK,QKBPS}, and will be convenient when we discuss the instanton corrected HK geometry. \\
    
 In Section \ref{instcon} we consider the instanton corrections to the semi-flat HK geometry associated to the resolved conifold. We start by recalling the construction of the instanton corrected HK geometries from the physics literature \cite{GMN1}, in terms of certain variations of BPS structures $(M,\Gamma,Z,\Omega)$ (see \cite{BridgelandDT} or Definition \ref{varBPS}). We then restrict to the simpler case where the instanton corrections are mutually local (see Definition \ref{defmutloc}), which is the relevant case for the resolved conifold. A mathematical treatment of the mutually local case can be found in \cite[Section 3]{QKBPS}. By adapting some results of \cite{QKBPS} for the case with flavor charges, one then gets an explicit form for the instanton corrected HK metric associated to the resolved conifold. See  Theorem \ref{theorem1} and Corollary \ref{HKexplicit}.\\
    
Let $(N,g_{N},\omega_1,\omega_2,\omega_3)$ be the HK manifold associated to the resolved conifold, with $\omega_{\alpha}$ denoting the K\"{a}hler forms. We end Section \ref{instcon} by comparing the HK structure  $(N,g_{N},\omega_1,\omega_2,\omega_3)$ with the so-called Ooguri-Vafa hyperk\"{a}hler metric (see \cite{NewCHK,LCS,OV}). The main result of this section is the following theorem:\\

\textbf{Theorem} \ref{ovsmoothing}: Let $(N,g_{N},\omega_1,\omega_2,\omega_3)$ be the (complex $2$-dimensional) HK manifold associated to the resolved conifold, and $t$ the special coordinate of Section \ref{ASKgeometry} of the ASK manifold associated to the resolved conifold.  Furthermore, let $(N^{\text{ov}},g^{\text{ov}},\omega_1^{\text{ov}},\omega_2^{\text{ov}},\omega_3^{\text{ov}})$ be the Ooguri-Vafa HK manifold centered at $t=0$ with cutoff $\Lambda=\frac{i}{2\pi}$. Then

\begin{equation}
    \omega_{\alpha}=\omega_{\alpha}^{\text{ov}}+\eta_{\alpha}, \;\;\; \text{for} \;\;\; \alpha=1,2,3\,,
\end{equation}
where $\eta_{\alpha}$ extend smoothly over the locus $t=0$. In particular, the HK structure of $(N,g_{N},\omega_1,\omega_2,\omega_3)$ smoothly extends over the points in the locus $t=0$ where the forms $\omega_{\alpha}$ remain non-degenerate.\\

This kind of Ooguri-Vafa-like smoothing of the semi-flat HK structure via instanton corrections was conjectured in \cite[Section 7]{NewCHK} for a certain general class of HK metrics, so Theorem \ref{ovsmoothing} provides an instance in which the conjecture holds, beyond the well-understood case of the Ooguri-Vafa space \cite{LCS}. \\

The HK structure $(N,g_{N},\omega_1,\omega_2,\omega_3)$ of the resolved conifold can be described in terms of two holomorphic twistor coordinates $\mathcal{X}_{\beta^{\vee}}(x,\zeta)$ and $\mathcal{X}_{\beta}(x,\zeta)$, in the sense that the $\mathcal{O}(2)$-twisted family of holomorphic symplectic forms associated to its twistor space is given by
    
    \begin{equation}
        \zeta\varpi(\zeta)\otimes \partial_{\zeta}=\frac{\zeta}{4\pi^2}\mathrm{d}\log( \mathcal{X}_{\beta^{\vee}}(\zeta))\wedge \mathrm{d}\log (\mathcal{X}_{\beta}(\zeta)) \otimes \partial_{\zeta}\,, \;\;\;\; \zeta \in \mathbb{C}\subset \mathbb{C}P^1.
    \end{equation}
    where $\mathrm{d}$ does not differentiate along the $\zeta$ direction.\\
    
    In Section \ref{confRH} we consider a certain conformal limit (see \cite{Gaiotto:2014bza}) of $\mathcal{X}_{\beta^{\vee}}$ and $\mathcal{X}_{\beta}$. While the conformal limit of $\mathcal{X}_{\beta}$ always exists, the one for $\mathcal{X}_{\beta^{\vee}}$ gives only a conditionally convergent expression. After specifying how to sum the conditionally convergent expression (see (\ref{condconv})), we will show that it satisfies properties similar to the ones required in the Riemann-Hilbert problem considered in \cite{BridgelandCon}. While Brideland's Riemann-Hilbert problem requires certain asymptotic conditions along half-planes (see (RH2) and (RH3) from Section \ref{RHprob}), our methods only show that the asymptotics conditions hold on certain smaller sectors determined by BPS rays. Nevertheless, this suggests the following conjecture (see also Remark \ref{endremark}):\\
    
    \textbf{Conjecture} \ref{Theorem2}: fix $t\in \mathbb{C}^{\times}$ with $\text{Im}(t)> 0$. Let $\mathcal{X}_{\beta^{\vee}}(t,\lambda)$ denote the specified  convergent expression of the conformal limit of $\mathcal{X}_{\beta^{\vee}}(x,\zeta)$, and let $\Phi_{\beta^{\vee}}(v,w,\lambda)$ be obtained via Definition \ref{Bsol} in terms of Bridgeland's solution to the Riemann-Hilbert problem of Section \ref{RHprob}. Then

\begin{equation}
    \mathcal{X}_{\beta^{\vee}}(t,\lambda)=\exp(2\pi i Z_{\beta^{\vee}}(t)/\lambda)\Phi_{\beta^{\vee}}(t,1,\lambda)\,,
\end{equation}
where $Z_{\beta^{\vee}}$ is the central charge evaluated on the charge $\beta^{\vee}$. In particular, if $\lambda$ is on the sector between $i\mathbb{R}_{+}t$ and $i\mathbb{R}_{+}(t-1)$ the following holds:

\begin{equation}
    \mathcal{X}_{\beta^{\vee}}(t,\lambda)=H(t|1,-\lambda)e^{Q_H(t|1,-\lambda)+2\pi i Z_{\beta^{\vee}}(t)/\lambda}\,,
\end{equation}
where $H(t|\omega_1,\omega_2)$ is the Faddeev quantum dilogarithm, and
\begin{equation}
    Q_H(t|\omega_1,\omega_2):=-\frac{\omega_1}{2\pi i\omega_2}\text{Li}_2(e^{2\pi it/\omega_1})-\frac{1}{2}\log(1-e^{2\pi i t/\omega_1})+\frac{\pi}{12}\frac{\omega_2}{\omega_1}\,.
\end{equation}


\section{The resolved conifold and an associated affine special K\"{a}hler geometry}\label{resASK}

\subsection{The resolved conifold}\label{rescon}

The conifold singularity refers to a singular point in a Calabi-Yau threefold that locally looks like
\begin{equation}
(x_1\, x_2 - x_3 \,x_4=0) \subset \mathbb{C}^4\,.
\end{equation}

The Calabi-Yau threefold given by the total space of the following rank two bundle over the projective line:
\begin{equation}
X := \mathcal{O}(-1) \oplus \mathcal{O}(-1) \rightarrow \mathbb{P}^1\,,
\end{equation}
corresponds to the resolution of the conifold singularity in $\mathbb{C}^4$ and is known as the resolved conifold. $C$ contains a unique compact curve, the zero section $C\simeq \mathbb{P}^1\subset X$. It defines a class $\beta= \left[ C\right] \in H_{2}(X,\mathbb{Z})$.  Furthermore, let $\delta$ be the generator of $H_0(X,\mathbb{Z})$. We now consider the lattice
\begin{equation}
\Gamma= \mathbb{Z}\cdot \delta + \mathbb{Z}\cdot \beta + \mathbb{Z} \cdot \beta^{\vee}\,,
\end{equation}
with pairing
\begin{equation}
        \langle \delta, \beta \rangle=0, \;\;\; \langle \delta, \beta^{\vee} \rangle=0,\;\;\; \langle \beta^{\vee}, \beta \rangle=1 \,. 
    \end{equation}
We think of $\beta^{\vee}$ as the generator in $H_4(X_c,\mathbb{Z})$, dual to $\beta \in H_{2}(X_c,\mathbb{Z})$, where $X_c$ refers to a compact CY geometry which contains the resolved conifold in a suitable limit. Alternatively one may think of $\beta^{\vee}$ as a regularized four-cycle class.\\

Let $B\in H^{2}(X,\mathbb{R})/H^{2}(X,\mathbb{Z})$ be the B-field and $\omega$ be the K\"ahler form  and $\omega_{\mathbb{C}}=B+i \omega$ the complexified K\"ahler form. We define
\begin{equation}
t= \int_{\beta} B+ i \omega\,,
\end{equation}
as well as the periods\footnote{The name periods is motivated by mirror symmetry and discussed in the appendix \ref{appendix:mirrorsymmetry}, these periods correspond to the quantum corrected volumes of the generators of $H_0(X,\mathbb{Z}),H_2(X,\mathbb{Z})$ as well as the regularized generator of $H_4(X,\mathbb{Z})$}

\begin{equation}\label{perioddef}
    \varpi^0=1\,, \quad \varpi^1=t\,, \quad \varpi^2=F_t\,,
\end{equation}
where 
$$F_t:= \frac{\partial}{\partial t} F_0(t)= \frac{1}{2}t^2 + \frac{1}{(2\pi i)^2}\textrm{Li}_2(q)\,, \quad q:=\exp(2\pi i t) \,.$$
Here, $F_0(t)$ is the genus zero, degree non-zero Gromov-Witten potential of the resolved conifold, given by
\begin{equation}
    F_0(t)= \frac{1}{6} t^3+ \frac{1}{(2\pi i)^3}\textrm{Li}_3(q)\,, \quad q:=\exp(2\pi i t)\,,
\end{equation}
and the polylogarithm is defined by for $|z|<1$
\begin{equation}
\textrm{Li}_s(z) = \sum_{n=0}^{\infty} \frac{z^n}{n^s}\, ,\quad s\in \mathbb{C}\,.
\end{equation}

For further motivation on the particular choice of periods, see Appendix \ref{appendix:mirrorsymmetry}.


\subsection{ASK geometries in terms of central charges}\label{ASKgeometrygeneral}

Before fully describing the affine special K\"{a}hler (ASK) geometry associated to the resolved conifold, we review a way of describing an ASK geometry in terms of the notion of central charges.   This perspective will be useful for the following sections, and can be found in \cite{GMN1,NewCHK,QKBPS}.\\

We start by considering a tuple $(M,\Gamma,Z)$, where

\begin{itemize}
    \item $M$ is a complex manifold. We denote $\text{dim}_{\mathbb{C}}(M)=r$.
    \item Charge lattice: $\Gamma\to M$ is a local system of lattices given as an extension
    
    \begin{equation}
        0 \to \Gamma_f \to \Gamma \xrightarrow[]{p} \Gamma_g \to0 \, ,
    \end{equation}
    where $\Gamma_f \to M$ is a trivial local system and $\Gamma_g\to M$ has rank $2r$. We assume that $\Gamma$ carries a skew, fiber-wise, parallel pairing $\langle - , - \rangle: \Gamma \times \Gamma \to \mathbb{Z}$ such that $\langle \gamma_f,\gamma \rangle=0$ for all $\gamma_f \in \Gamma_f$ and $\gamma \in \Gamma$. We then have an induced pairing $\langle - , - \rangle: \Gamma_g \times \Gamma_g \to \mathbb{Z}$, which we assume locally admits Darboux frames. Our convention will be that a local Darboux frame  $(\widetilde{\gamma}_i,\gamma^i)$ of $\Gamma_g$ satisfies $\langle \widetilde{\gamma}_i,\gamma^j\rangle= \delta_{i}^j$. 
    \item Central charge: $Z$ is a holomorphic section of $\Gamma^*\otimes \mathbb{C} \to M$, where $\Gamma^*$ denotes the dual of $\Gamma$. Given a local section $\gamma$ of $\Gamma|_U$ for $U\subset M$, we denote $Z_{\gamma}:=Z(\gamma):U\to \mathbb{C}$ the corresponding holomorphic function. If $\gamma_f$ is a section of $\Gamma_f$, we assume that $Z_{\gamma_f}$ is a constant function.
\end{itemize}

\begin{rem}\label{flavor}
We will frequently refer to elements of $\Gamma_f$ as flavor charges. Similarly, we will refer to $Z_{\gamma}$ for $\gamma \in \Gamma_f$ as a flavor period. On the other hand, the elements of $\Gamma_g$ are typically referred to as gauge charges in the physics literature. 
\end{rem}

\begin{dfn} \label{centralASK} A tuple $(M,\Gamma,Z)$ as before will be called a central charge ASK geometry if the following holds:

\begin{itemize}
    \item Given local sections $(\widetilde{\gamma}_i,\gamma^i)$ of $\Gamma$ projecting to a Darboux frame of $\Gamma_g$, the 1-forms $dZ_{\gamma^i}$ (or $dZ_{\widetilde{\gamma}_i}$) give a local frame of $T^*M$. In particular, $(Z_{\gamma^i})$ (or $(Z_{\widetilde{\gamma}_i})$) give local coordinates on $M$.
    \item By using the identification $\Gamma_g \cong \Gamma_g^*$ given by $\gamma \to \langle \gamma, - \rangle \in \Gamma_g^*$, we consider the induced pairing on $\Gamma_g^*$, and extend it $\mathbb{C}$-bilinearly to a ($\mathbb{C}$-valued) pairing $\langle -, - \rangle$ on  $\Gamma_g^*\otimes \mathbb{C}$. With respect to this pairing we have
    
    \begin{equation}
        \langle \mathrm{d}Z \wedge \mathrm{d}Z \rangle=0 \,.
    \end{equation}
    \begin{rem}
    In the above, we think of $dZ$ as a 1-form with values in $\Gamma^*_g\otimes \mathbb{C}$, due to the fact that $Z_{\gamma_f}$ is constant for a section $\gamma_f$ of $\Gamma_f$.
    \end{rem}
    \item The two-form $\omega:=\frac{1}{4}\langle \mathrm{d}Z \wedge \mathrm{d}\overline{Z} \rangle$ is non-degenerate. In particular, if $J$ denotes the complex structure of $M$, then $g:=\omega(-,J-)$ is a (possibly-indefinite) metric on $M$. 
\end{itemize}

\end{dfn}
\begin{rem} The flavor data of $\Gamma_f$ and $Z|_{\Gamma_f}$ do not have any influence in the ASK geometry defined by $(M,\Gamma,Z)$. The reason to include it will become more clear when we also consider the data of BPS indices $\Omega(\gamma)$ and the corresponding instanton corrected HK geometry \cite{NewCHK}.

\end{rem}

\begin{prop}\label{propASK} Let $(M,\Gamma,Z)$ be a central charge ASK geometry. Then $(M,\omega)$ is an ASK manifold with possibly indefinite signature. 

\end{prop}

\begin{proof} 
Let $(\widetilde{\gamma}_i,\gamma^i)$ be local sections of $\Gamma$ projecting to a local Darboux frame of $\Gamma_g\to M$ over $U\subset M$. By using the first condition in Definition \ref{centralASK}, we can write 

\begin{equation}
    \mathrm{d}Z_{\widetilde{\gamma}_i}=\tau_{ij}\mathrm{d}Z_{\gamma^j} 
\end{equation}
for $\mathbb{C}$-valued functions $\tau_{ij}$ on $U\subset M$.\\

The second condition in Definition \ref{centralASK} then implies

\begin{equation}\label{symmtau}
    0=\langle \mathrm{d}Z \wedge \mathrm{d}Z \rangle=2\mathrm{d}Z_{\widetilde{\gamma}_i}\wedge \mathrm{d}Z_{\gamma^i}=2\tau_{ij}\mathrm{d}Z_{\gamma^j}\wedge \mathrm{d}Z_{\gamma^i}=2\sum_{i<j}(\tau_{ij}-\tau_{ji})\mathrm{d}Z_{\gamma^i}\wedge \mathrm{d}Z_{\gamma^j} \, ,
\end{equation}
and hence $\tau_{ij}=\tau_{ji}$. From (\ref{symmtau}) we also find that $Z_{\widetilde{\gamma}_i}\mathrm{d}Z_{\gamma^i}$ is closed. Hence, by possibly shrinking $U$, we can find $\mathfrak{F}:U \to \mathbb{C}$ such that

\begin{equation}
    \frac{\partial \mathfrak{F}}{\partial Z_{\gamma^i}}=Z_{\widetilde{\gamma}_i}, \;\;\;\;\; \frac{\partial^2 \mathfrak{F}}{\partial Z_{\gamma^i} \partial Z_{\gamma^j}}=\tau_{ij} \, .
\end{equation}

Finally, the third condition in Definition \ref{ASKgeometry} implies that

\begin{equation}
    \omega=\frac{1}{4}\langle \mathrm{d}Z \wedge \mathrm{d}\overline{Z} \rangle=\frac{1}{4}(\mathrm{d}Z_{\widetilde{\gamma}_i}\wedge \mathrm{d}\overline{Z}_{\gamma^i}-\mathrm{d}Z_{\gamma^i}\wedge \mathrm{d}\overline{Z}_{\widetilde{\gamma}_i})=\frac{i}{2}\text{Im}(\tau_{ij})\mathrm{d}Z_{\gamma^i}\wedge \mathrm{d}\overline{Z}_{\gamma^j}
\end{equation}
is a K\"{a}hler form for the complex structure $J$, corresponding to a possibly indefinite K\"{a}hler metric.\\

We then recover the usual local formulas of an ASK geometry. The function $\mathfrak{F}$ from above is a holomorphic prepotential describing locally the ASK geometry. We furthermore remark that $(Z_{\gamma^i})$ and $(Z_{\widetilde{\gamma}_i})$ are a conjugate system of holomorphic special coordinates for each choice of local sections of $\Gamma$ projecting to a Darboux frame $(\widetilde{\gamma}_i,\gamma^i)$ of $\Gamma_g$.
\end{proof}

\subsection{The ASK geometry of the resolved conifold in terms of central charges}\label{ASKgeometry}

Our aim here will be to specify a central charge ASK geometry $(M,\Gamma,Z)$ associated to the resolved conifold. In order to define this tuple, we take into account the periods defined in (\ref{perioddef}). From them we see that we have a natural choice of flavor period, given by $\varpi^0$; and two natural choices of holomorphic prepotentials given by $\pm F_0(t)$, where $t=\varpi^1$. We therefore want to define $(M,\Gamma,Z)$ in such a way that $Z$ encodes the periods $\varpi^1=t$ and $\varpi^2=F_t$ \footnote{The additional third period $\varpi^3=2F_0-tF_t$ in (\ref{periods}) does not give any new information, since it again contains $F_0$ and its derivative $F_t$.}. \\

We define $(M,\Gamma,Z)$ as follows:
 
 \begin{itemize}
     \item $M\subset \mathbb{C}^{\times}:=\mathbb{C}-\{0\}$ is a complex manifold of dimension $1$ given by
     \begin{equation}
        M:=\{ t\in \mathbb{C}^{\times} \;\; |  \;\; |\text{Re}(t)|<\frac{1}{2}, \;\;\;\; 2\text{Re}(e^{2\pi it}) \neq 1\} \, .
    \end{equation}
    The first constraint is due to the principal branch of the log in $\varpi^1=t=\frac{1}{2\pi i}\log(z)$ in Appendix \ref{appendix:mirrorsymmetry}. We could in principle  extend $t$ past $|\text{Re}(t)|<\frac{1}{2}$, but for simplicity we will consider only this restricted region. On the other hand the second constraint $2\text{Re}(e^{2\pi it}) \neq 1$ will be required in order to have $\text{Im}(\tau)\neq 0$, and hence define a non-degenerate $2$-form $\omega$.  
    \item In order to define $\Gamma \to M$, we first consider the open subset $M_0\subset M$ defined by
    \begin{equation}
        M_0:=\{ t\in M \;\; |  \;\;  t\not \in i \mathbb{R}_{\leq 0}\} \, .
    \end{equation}
    This is the region inside $M$ where the principal branch of $\text{Li}_2(e^{2\pi it})$ in $\varpi^2$ is defined.\\
    
    We define $\Gamma|_{M_0} \to M_0$ as a trivial local system of rank $3$, having a global trivialization by $(\delta,\beta,\beta^{\vee})$ and pairing defined by 
    \begin{equation}
        \langle \delta, \beta \rangle=0, \;\;\; \langle \delta, \beta^{\vee} \rangle=0,\;\;\; \langle \beta^{\vee}, \beta \rangle=1  \,.
    \end{equation}
    We now define $\Gamma \to M$ by declaring $\delta$ and $\beta$ to be global sections, while  $\beta^{\vee}$ has the following jump as $t$ goes through the ray $i\mathbb{R}_{\leq 0}$ anti-clockwise
    \begin{equation}
        \beta^{\vee}\to \beta^{\vee}+\beta \,.
    \end{equation}
    This jump preserves the pairing, so $\langle -, - \rangle$ extends to a pairing on $\Gamma \to M$. Clearly $\Gamma_f=\mathbb{Z}\delta$, and $\Gamma_g$ has $(\beta^{\vee},\beta)$ as a local Darboux frame for the induced pairing.\\
    
    \item To define $Z:M \to \Gamma^*\otimes \mathbb{C}$ we first define $Z:M_0\to \Gamma^*\otimes \mathbb{C}$ by
    \begin{equation}
        \begin{split}
        Z_{\delta}&:=\varpi^0=1 \,,\\
         Z_{\beta}&:=\varpi^1=t \,,\\
         Z_{\beta^{\vee}}:=-\varpi^2=-\frac{1}{(2\pi i)^2} \Big( \frac{1}{2} \Big(&\log(e^{2\pi i t})\Big)^2 + \textrm{Li}_2(e^{2\pi i t})\Big)\,,\\
        \end{split}
    \end{equation}
    where $\varpi^0$, $\varpi^1$ and $\varpi^2$ are the periods from before (\ref{perioddef}). 
    \begin{rem}
    We have picked $-\varpi^2$ instead of $\varpi^2$ in order for the ASK geometry to be positive definite near $t=0$, as we will see below. This choice will also be important for the relation with the Ooguri-Vafa metric in Section \ref{ovsection}.
    \end{rem}
    Clearly $Z_{\delta}$ and $Z_{\beta}$ extend to all of $M$. For $Z_{\beta^{\vee}}$ we notice that as $t$ crosses the ray $i\mathbb{R}_{\leq 0}$ anti-clockwise, then
    \begin{equation}
        Z_{\beta^{\vee}}\to Z_{\beta^{\vee}}+\frac{\log(e^{2\pi i t})}{2\pi i}=Z_{\beta^{\vee}}+t=Z_{\beta^{\vee}}+Z_{\beta} \, .
    \end{equation}
    Since this matches jump of $\beta^{\vee}$, we see that $Z$ extends to a global section of $\Gamma^*\otimes \mathbb{C}\to M$.
 \end{itemize}

A holomorphic prepotential $\mathfrak{F}:M \to \mathbb{C}$ for the ASK geometry is given by

\begin{equation}
    \mathfrak{F}(t) :=-F_0(t)=-\frac{1}{(2\pi i)^3}\left( \frac{1}{3!} (\log(e^{2\pi i t}))^3+ \textrm{Li}_3(e^{2\pi it})\right)\, .
\end{equation}
The  K\"{a}hler form for the ASK geometry is then given by the usual formula

\begin{equation}
    \omega=\frac{i}{2}\text{Im}(\tau)\mathrm{d}t\wedge \mathrm{d}\overline{t}, \;\;\;\;\;\;\; \tau:=\frac{\partial^2 \mathfrak{F}}{\partial t^2}=\frac{1}{2\pi i}\Big(\log(1-e^{2\pi it})-\log(e^{2\pi it})\Big) \,.
\end{equation}
Furthermore, we remark again that the constraint $2\text{Re}(e^{2\pi i t}) \neq 1$ in the definition of $M$ implies that 

\begin{equation}
    \text{Im}(\tau)=-\frac{1}{2\pi}\log\Big|\frac{1-e^{2\pi it}}{e^{2\pi it}}\Big|\neq 0 \,,
\end{equation}
and hence $\omega$ is non-degenerate on $M$. Notice that $2\text{Re}(e^{2\pi it})\neq 1$ divides $M$ into two components $M=M_{+}\cup M_{-}$. $M_{-}$ contains the ray $i\mathbb{R}_{>a}$ for $a>0$ sufficiently big, while $M_{+}$ contains all $t$ sufficiently close to $0$ and the ray $i\mathbb{R}_{<0}$. Furthermore, it is easy to check that $\text{Im}(\tau)<0$ on $M_-$ while $\text{Im}(\tau)>0$ on $M_+$. We therefore obtain

\begin{prop}
$(M,\Gamma,Z)$ is a central charge ASK geometry. If $g_{sK}:=\omega(-,J-)$ denotes the ASK metric, then $g_{sK}$ is a positive definite metric on $M_{+}$,  and a negative definite metric on $M_{-}$.
\end{prop}
\begin{proof}
The first condition of Definition \ref{centralASK} follows from the fact that in the local Darboux frame $(\beta^{\vee},\beta)$ of $\Gamma_g$ we have $\mathrm{d}Z_{\beta}=\mathrm{d}t$, while $\langle \mathrm{d}Z\wedge \mathrm{d}Z\rangle=2\mathrm{d}Z_{\beta^{\vee}}\wedge \mathrm{d}Z_{\beta}=0$ follows from $\mathrm{d}Z_{\beta^{\vee}}=\tau \mathrm{d}Z_{\beta}=\tau \mathrm{d}t$. The non-degeneracy condition of the K\"{a}hler form and the statements about the signatures then follow from our previous arguments about $\text{Im}(\tau)$.
\end{proof}

\subsection{The semi-flat HK geometry associated to an ASK manifold}

Given an ASK manifold $(M,\omega)$ of signature $(n,m)$, one can always associate an hyperk\"{a}hler manifold $(\mathcal{N},g^{\text{sf}},I_1,I_2,I_3)$ of signature $(2n,2m)$ via the rigid c-map construction \cite{TypeIIgeometry,Freed,ACD}. Here $\mathcal{N}$ is the total space of a torus bundle $\pi:\mathcal{N} \to M$, and the metric $g^{\text{sf}}$ is usually known as the semi-flat metric, since it restricts to a flat metric on the fibers of $\pi$. 
Below we recall a convenient twistor space description of this metric \cite{GMN1,NewCHK}, in the case that the $(M,\omega)$ is described via $(M,\Gamma,Z)$. This description via twistor Darboux coordinates will be convenient for the later sections.\\

Given the charge lattice $\Gamma \to M$, we fix a homomorphism $\theta_{f}:\Gamma_f\to \mathbb{R}/2\pi \mathbb{Z}$, and define $\pi: \mathcal{N}(\theta_f)\to M$ as the bundle with fiber

\begin{equation}
    \mathcal{N}(\theta_f)_p:=\{\theta: \Gamma_p \to \mathbb{R}/2\pi \mathbb{Z} \;\; | \;\; \theta_{\gamma+\gamma'}=\theta_{\gamma}+\theta_{\gamma'}, \;\;\; \theta|_{\Gamma_f}=\theta_f \} \,.
\end{equation}
We will consider the evaluation map $\theta: \mathcal{N}(\theta_f) \to \Gamma^*\otimes \mathbb{R}/2\pi \mathbb{Z}$ and denote by $\theta_{\gamma}: \mathcal{N}(\theta_f) \to \mathbb{R}/2\pi \mathbb{Z}$ the corresponding evaluation at $\gamma \in \Gamma$.\\

Given a section $\gamma$  of $\Gamma$ over $U\subset M$ and $\zeta \in \mathbb{C}^{\times}$, one can define $\mathcal{X}_{\gamma}^{\text{sf}}:\pi^{-1}(U)\times \mathbb{C}^{\times} \to \mathbb{C}^{\times}$ by

\begin{equation}\label{sftwistcoord}
    \mathcal{X}_{\gamma}^{\text{sf}}(\zeta):=\exp\Big(\pi \frac{Z_{\gamma}}{\zeta}+i\theta_{\gamma}+\pi \zeta \overline{Z}_{\gamma}\Big) \,,
\end{equation}
where the pullbacks of $Z_{\gamma}$ by $\pi$ are suppressed (the $\pi$'s in the above expression refer to the mathematical constant). 
We will refer to  $\mathcal{X}_{\gamma}^{\text{sf}}(\zeta)$ as the semi-flat twistor coordinates. \\

The $\mathcal{O}(2)$-twisted family of holomorphic symplectic forms describing the semiflat HK geometry on $\mathcal{N}(\theta_f)$ is then given by

\begin{equation}\label{twistholsym}
    \zeta\varpi^{\text{sf}}(\zeta)\otimes \partial_{\zeta}=\frac{\zeta}{8\pi^2}\langle \mathrm{d}\log(\mathcal{X}^{\text{sf}}(\zeta))\wedge \mathrm{d}\log(\mathcal{X}^{\text{sf}}(\zeta))\rangle \otimes \partial_{\zeta} \,,
\end{equation}
where $\mathrm{d}$ differentiates along only the $\mathcal{N}(\theta_f)$ directions\footnote{ By this statement we mean the directions of the total space of the bundle $\mathcal{N}(\theta_f)\to M$. In other words $\mathrm{d}$ in \eqref{twistholsym} differentiates in all directions expect in the twistor parameter $\zeta \in \mathbb{C}^{\times}$.} and $\zeta \in \mathbb{C}\subset \mathbb{C}P^1$ is a linear holomorphic coordinate. In particular, one has the following expansion near $\zeta=0$:

\begin{equation}
    \zeta \varpi^{\text{sf}}(\zeta)=-\frac{i}{2}(\omega_1^{\text{sf}}+i\omega_2^{\text{sf}}) + \zeta\omega_3^{\text{sf}} -\frac{i}{2}\zeta^2 (\omega_1^{\text{sf}}-i\omega_2^{\text{sf}}) \,,
\end{equation}
where $\omega_{\alpha}^{\text{sf}}$ for $\alpha=1,2,3$ correspond to a triple of K\"{a}hler forms for the semi-flat hyperk\"{a}hler structure. \\

The $\omega_{\alpha}^{\text{sf}}$ have the following particularly simple formulas in terms of the central charge $Z$ and the evaluation map $\theta$ (below, the pullback of $Z$ to $\mathcal{N}(\theta_f)$ is suppressed):

\begin{equation}    
    \begin{split}
    \omega_1^{\text{sf}}+i\omega_2^{\text{sf}}&=-\frac{1}{2\pi}\langle \mathrm{d}Z\wedge \mathrm{d}\theta \rangle \;,\\
    \omega_3^{\text{sf}}&=\frac{1}{4}\langle \mathrm{d}Z\wedge \mathrm{d}\overline{Z}\rangle - \frac{1}{8\pi^2}\langle \mathrm{d}\theta \wedge \mathrm{d}\theta \rangle\,.\\
    \end{split}
\end{equation}

In order to write down the metric, we will use the following formulas. With respect to local sections $(\widetilde{\gamma}_i,\gamma^i)$ of $\Gamma$ projecting to a local Darboux frame of $\Gamma_g$, we define $\tau_{ij}$ via

\begin{equation}
    \mathrm{d}Z_{\widetilde{\gamma_i}}=\tau_{ij}\mathrm{d}Z_{\gamma^j} \,,
\end{equation}
and define

\begin{equation}
    N_{ij}:=\text{Im}(\tau_{ij}), \;\;\;\;\; W_i:=\mathrm{d}\theta_{\widetilde{\gamma}_i}-\tau_{ij}\mathrm{d}\theta_{\gamma^j} \,.
\end{equation}

The corresponding semiflat HK metric $g^{\text{sf}}$ on $\mathcal{N}(\theta_f)$ then has the following local form

\begin{equation}\label{semiflatHKmetric}
    g^{\text{sf}}=N_{ij}\mathrm{d}Z_{\gamma^i}\mathrm{d}\overline{Z}_{\gamma^j} + \frac{1}{4\pi^2}N^{ij}W_i\overline{W}_j \,.
\end{equation}
This formula follows from

\begin{equation}
    \begin{split}
        \omega_3^{\text{sf}}&=\frac{1}{4}\langle \mathrm{d}Z\wedge \mathrm{d}\overline{Z}\rangle - \frac{1}{8\pi^2}\langle \mathrm{d}\theta \wedge \mathrm{d}\theta \rangle \,,\\
        &=\frac{i}{2}N_{ij}\mathrm{d}Z_{\gamma^i}\wedge \mathrm{d}\overline{Z}_{\gamma^j}+\frac{i}{8\pi^2}N^{ij}W_i\wedge \overline{W}_j\,,
    \end{split}
\end{equation}
and the fact that $\mathrm{d}Z_{\gamma^i}$ and $W_i$ are of type $(1,0)$ with respect to the complex structure $I_3$.\\

Furthermore, if $N_{ij}$ has signature $(n,m)$, it is easy to check that $g^{\text{sf}}$ must have signature $(2n,2m)$.

\subsection{The semi-flat geometry associated to the resolved conifold}

Now let $(M,\Gamma,Z)$ be the tuple defining the ASK geometry associated to the resolved conifold in Section \ref{ASKgeometry}. \\

Fixing a homomorphism $\theta_f: \Gamma_f \to \mathbb{R}/2\pi \mathbb{Z}$ (which in this case is equivalent to fixing the value of $\theta_{\delta}$), we have the corresponding semi-flat metric $g^{\text{sf}}$ on the total space of $\mathcal{N}(\theta_f)\to M$. In this case, (\ref{semiflatHKmetric}) reduces to the following formula in the local Darboux frame $(\beta^{\vee},\beta)$ of $\Gamma_g$:
\begin{equation}
    \begin{split}
    g^{\text{sf}}&=\text{Im}(\tau)|\mathrm{d}Z_{\beta}|^2 + \frac{\text{Im}(\tau)^{-1}}{4\pi^2}|W|^2\,,\\
    &=\text{Im}(\tau)|\mathrm{d}t|^2 + \frac{\text{Im}(\tau)^{-1}}{4\pi^2}|\mathrm{d}\theta_{\beta^{\vee}}-\tau \mathrm{d}\theta_{\beta}|^2\,,\\
    \end{split}
\end{equation}
where
\begin{equation}
\tau=\frac{1}{2\pi i}\Big(\log(1-e^{2\pi it})-\log(e^{2\pi it})\Big) \,.
\end{equation}

Since $\text{Im}(\tau)>0$ on $M_{+}$ and $\text{Im}(\tau)<0$ on $M_{-}$, we get that the signature $g^{\text{sf}}$ is positive definite on $\mathcal{N}_{+}:=\pi^{-1}(M_+)$, and negative definite on $\mathcal{N}_{-}:=\pi^{-1}(M_{-})$.

\section{The associated instanton corrected hyperk\"{a}hler geometry}\label{instcon}

In this section, we wish to add the data of BPS indices $\Omega(\gamma)$ of the resolved conifold, to the central charge ASK geometry $(M,\Gamma,Z)$ associated to the resolved conifold in Section \ref{ASKgeometry}. From the tuple $(M,\Gamma,Z,\Omega)$ one can then define an ``instanton corrected" HK geometry (correcting the semi-flat HK geometry), which we seek to describe explicitly.\\

In Section \ref{instHKrev} we start by reviewing the construction of instanton corrected HK geometries from the physics literature \cite{GMN1, NewCHK}\footnote{For the corresponding constructions in the physics literature for quaternionic-K\"{a}hler geometry, see for example the reviews \cite{HMreview1,HMreview2} and references therein.}. We then specialize to the case where the instanton corrections are mutually local (see Definition \ref{defmutloc}), which will be the one needed for the instanton corrected HK geometry associated to the resolved conifold. A full mathematical treatment of the mutually local case was given in \cite[Section 3]{QKBPS}, so we will review some results of \cite{QKBPS}, and discuss some slight extensions to include flavor charges. \\

In Section \ref{instconifold} we specialize the general formulas of Section \ref{instHKrev} to the HK geometry associated to the resolved conifold, and give explicit formulas for the HK metric. Finally, in Section \ref{ovsection} we compare the HK metric associated to the resolved conifold near $t=0$ to the so-called Ooguri-Vafa metric (see \cite{LCS,NewCHK}). We will see that the HK structure associated to the resolved conifold admits an extension over the locus $t=0$, realizing a specific case of a conjecture of \cite[Section 7]{NewCHK}.

\subsection{Review of instanton corrected HK geometries}\label{instHKrev}

We start by recalling the notion of variation of BPS structures \cite{BridgelandDT}. 
\begin{dfn}\label{varBPS} A variation of BPS structures is given by a tuple $(M,\Gamma,Z,\Omega)$, where 

\begin{itemize}
    \item $M$ is a complex manifold. 
    \item $\Gamma \to M$ is a local system of lattices with a skew-symmetric, covariantly constant paring $\langle - , - \rangle: \Gamma \times \Gamma \to \mathbb{Z}$. As in Section \ref{ASKgeometrygeneral}, we refer to $\Gamma\to M$ as the charge lattice. 
    \item $Z$ is a holomorphic section of $\Gamma^*\otimes \mathbb{C} \to M$. We refer to $Z$ as the central charge. 
    \item $\Omega : \Gamma \to \mathbb{Z}$ is a function (of sets) satisfying $\Omega(\gamma)=\Omega(-\gamma)$ and the Kontsevich-Soibelman wall-crossing formula \cite{KS, BridgelandDT,NewCHK}. We refer to $\Omega$ as the BPS indices. 
\end{itemize}

Furthermore, the tuple $(M,\Gamma,Z,\Omega)$ should satisfy the following conditions:

\begin{itemize}
\item Support property: Let $\text{Supp}(\Omega):=\{\gamma \in \Gamma \;\; | \;\; \Omega(\gamma)\neq 0\}$. Given a compact set $K\subset M$ and a choice of covariantly constant norm $|\cdot |$ on $\Gamma|_K \otimes_{\mathbb{Z}}\mathbb{R} $, there is a constant $C>0$ such that for any  $\text{Supp}(\Omega)\cap \Gamma|_K$ the following holds:
    \begin{equation} \label{supportproperty}
        |Z_{\gamma}|>C|\gamma| \,.
    \end{equation}
    
    \item Convergence property: for each $p\in M$, there is an $R>0$ such that
    \begin{equation}
        \sum_{\gamma \in \Gamma_p}|\Omega(\gamma)|e^{-R|Z_{\gamma}|}<\infty \,,
    \end{equation}
    where $\Gamma_p$ denotes the fiber of $\Gamma$ over $p$.
\end{itemize}
\end{dfn}
\begin{rem}
As a consequence of the BPS indices $\Omega$ obeying the wall-crossing formula, the following holds:
\begin{itemize}
    \item  Consider the real codimension $1$ subset $\mathcal{W}\subset M$ (the so called ``wall" or ``walls") defined by
    \begin{equation}\label{wall}
            \mathcal{W}:=\{p\in M \;\; | \;\; \exists \gamma, \gamma' \in  \Gamma_p, \;\; \langle \gamma, \gamma' \rangle\neq 0, \;\; Z_{\gamma}/Z_{\gamma'}\in \mathbb{R}_{>0} \} \,.
    \end{equation}
   Then for a local section $\gamma$ of $\Gamma$, the BPS index $\Omega(\gamma)$ is locally constant on $M\backslash \mathcal{W}$. Furthermore, the $\Omega(\gamma)$ jumps across $\mathcal{W}$, and the discontinuity is determined by the Kontsevich-Soibelman wall-crossing formula.
   \item $\Omega$ is monodromy invariant. That is, if $\gamma$ has monodromy $\gamma \to \gamma'$ around a loop, then $\Omega(\gamma)=\Omega(\gamma')$.
    
\end{itemize}
\end{rem}

We will focus on the following types of variations of BPS structures:

\begin{dfn}\label{HKBPS} A hyperk\"{a}hler (HK) variation of BPS structures is a variation of BPS structures $(M,\Gamma,Z,\Omega)$ such that:

\begin{itemize}
    \item $(M,\Gamma,Z)$ is a central charge ASK geometry as in Definition \ref{centralASK}.
    \item Strong convergence property: for any $R>0$ the series
    \begin{equation}\label{convergenceproperty}
        \sum_{\gamma \in \Gamma}|\Omega(\gamma)|e^{-R|Z_{\gamma}|}
    \end{equation}
    converges normally on compact subsets of $M$.
\end{itemize}
\end{dfn}
\begin{rem}
The strong convergence property appeared in \cite{QKBPS}. This condition was used to ensure smoothness of the triple of instanton corrected hyperk\"{a}hler forms in the simpler case of mutually local variations of BPS structures (see Definition \ref{defmutloc}). There may be a weaker condition that also ensures this, but for the purposes of this work this condition suffices. 
\end{rem}
In \cite{GMN1,NewCHK}, it is argued that one can construct an ``instanton corrected" hyperk\"{a}hler structure out of an HK variation of BPS structures.  We now briefly review the main points of this construction, and then restrict to the simpler mutually local case. As we remarked before, a mathematical treatment for the mutually local case can be found in \cite[Section 3]{QKBPS}.

\subsubsection{GMN construction of instanton corrected HK geometries}

Let $(M,\Gamma,Z,\Omega)$ be an HK variation of BPS structures and fix a homomorphism $\theta_f: \Gamma_f \to \mathbb{R}/2\pi \mathbb{Z}$. We first define the bundle $\pi: \mathcal{M}(\theta_f)\to M$ of ``twisted" unitary characters as the bundle with fiber over $p\in M$ given by

\begin{equation}
    \mathcal{M}(\theta_f)_p:=\{ \theta: \Gamma_p \to \mathbb{R}/2\pi \mathbb{Z} \;\;\; | \;\;\; \theta_{\gamma + \gamma'}=\theta_{\gamma}+\theta_{\gamma'}+\pi \langle \gamma, \gamma' \rangle,\;\;\;\; \theta|_{\Gamma_{f}}=\theta_f\} \,.
\end{equation}

One should then find locally defined functions $\mathcal{X}_{\gamma}:U\subset \mathcal{M}(\theta_f)\times \mathbb{C}^{\times}\to \mathbb{C}^{\times}$, labeled by local sections $\gamma$ of $\Gamma|_{\pi(U)}$, and satisfying the GMN equations\footnote{These equations are also known as the TBA equations, due to the similarity to the Thermodynamic Bethe Ansatz equations. See for example \cite[Appendix E]{GMN1}.}:
    
    \begin{equation}\label{GMNeq}                    \mathcal{X}_{\gamma}(x,\zeta)=
    \mathcal{X}^{\text{sf}}_{\gamma}(x,\zeta)\exp \Big[-\frac{1}{4\pi i}\sum_{\gamma'\in \Gamma_{\pi(x)}}\Omega(\gamma')\langle \gamma,\gamma' \rangle\int_{\mathbb{R}_{-}Z_{\gamma'}}\frac{d\zeta'}{\zeta'}\frac{\zeta'+\zeta}{\zeta'-\zeta}\log(1-\mathcal{X}_{\gamma'}(x,\zeta'))\Big] \,,
    \end{equation}
    where $\mathcal{X}^{\text{sf}}_{\gamma}(x,\zeta)$ are the semi-flat twistor coordinates given by formula (\ref{sftwistcoord})  and $(x,\zeta)\in \mathcal{M}(\theta_f)\times \mathbb{C}^{\times}$.\\
    
For a fixed $x \in \mathcal{M}(\theta_f)$ the functions $\mathcal{X}_{\gamma}(x,\zeta)$ have discontinuities along rays $\mathbb{R}_{-}Z_{\gamma'}$ for which $\gamma' \in \text{Supp}(\Omega)$ (the so-called ``BPS rays"). Furthermore, the twist on unitary characters and the GMN equations also ensures that they satisfy the identity

\begin{equation}
    \mathcal{X}_{\gamma}\mathcal{X}_{\gamma'}=(-1)^{\langle \gamma, \gamma' \rangle}\mathcal{X}_{\gamma+\gamma'} \,,
\end{equation}
which is important for the wall-crossing formalism of Kontsevich-Soibelman \cite{KS}.\\

Finally, the functions $\mathcal{X}_{\gamma}(x,\zeta)$ are used to define a $\mathbb{C}^{\times}$-family of $2$-forms on $\mathcal{M}(\theta_{f})$ by the formula
    \begin{equation}\label{O2holsym}
    \varpi(\zeta)=\frac{1}{8\pi^2 }\langle \mathrm{d}\log(\mathcal{X}(\zeta))\wedge \mathrm{d}\log(\mathcal{X}(\zeta))\rangle \,,
    \end{equation}
where the exterior derivative only differentiates in the $x\in \mathcal{M}(\theta_{f})$ directions. The crucial points are that the discontinuities in $\zeta$ of $\mathcal{X}_{\gamma}(x,\zeta)$ are such that they leave $\varpi(\zeta)$ invariant; and that the wall-crossing formula ensures that $\varpi(\zeta)$ actually extends over $\mathcal{W}\subset M$, where the BPS indices jump.\\

It is then argued that there is an HK twistor space structure on $\mathcal{M}(\theta_f)\times \mathbb{C}P^1$, whose $\mathcal{O}(2)$-twisted family of holomorphic symplectic forms is given by
\begin{equation}
    \zeta \varpi (\zeta)\otimes \partial_{\zeta} \,,
\end{equation}
with $\varpi(\zeta)$ as in (\ref{O2holsym}).\\

The triple of K\"{a}hler forms for the HK structure on $\mathcal{M}(\theta_f)$ can then be extracted by expanding $\zeta(\varpi(\zeta))$ near $\zeta=0$. Indeed, one obtains

\begin{equation}
        \zeta\varpi(\zeta)=-\frac{i}{2}(\omega_1+i\omega_2) + \zeta\omega_3 -\frac{i}{2}\zeta^2(\omega_1 -i\omega_2) \,,
\end{equation}
where $\omega_{\alpha}$ for $\alpha=1,2,3$ correspond to the triple of K\"{a}hler forms.

\subsubsection{The mutually local case}
For the case of interest for the paper, we will need to restrict the previous construction to the simpler case of mutually local variations of BPS structures, defined as follows.

\begin{dfn}\label{defmutloc} A variation of BPS structures  $(M,\Gamma,Z,\Omega)$ is called mutually local if $\gamma, \gamma' \in \text{Supp}(\Omega)$ implies that $\langle \gamma, \gamma' \rangle=0$.
\end{dfn}

\begin{rem}\leavevmode
\begin{itemize}
    \item In the mutually local case we have no wall-crossing phenomena. In particular, given a local section $\gamma$ of $\Gamma$, $\Omega(\gamma)$ is locally constant on $M$.
    \item We also remark that since there is no wall-crossing, one can work instead on $\pi: \mathcal{N}(\theta_{f})\to M$ whose fiber over $p \in M$ is given by 
\begin{equation}
    \mathcal{N}(\theta_f)_p:=\{ \theta: \Gamma_p \to \mathbb{R}/2\pi \mathbb{Z} \;\;\; | \;\;\; \theta_{\gamma + \gamma'}=\theta_{\gamma}+\theta_{\gamma'},\;\;\;\; \theta|_{\Gamma_{f}}=\theta_f\} \,.
\end{equation}
The total spaces $\mathcal{M}(\theta_f)$ and $\mathcal{N}(\theta_f)$ can be locally identified (non-canonically), but globally they might differ topologically (see the discussion in \cite{GMN1} about this issue). In the following, we work with $\mathcal{N}(\theta_f)$ for simplicity, and for easier comparison with some results of \cite{QKBPS}.
\end{itemize}
\end{rem}

 In the mutually local case the GMN equations reduce to integral formulas, and one can explicitly write down $\varpi(\zeta)$. To see why this is so, let $p:\Gamma \to \Gamma_g$ denote the projection. Then if $\gamma, \gamma' \in p(\text{Supp}(\Omega))$, we must have $\langle\gamma, \gamma'\rangle=0$. In particular, one can find local sections $\{\widetilde{\gamma}_i, \gamma^i\}$ such that  $(p(\widetilde{\gamma}_i),p(\gamma^i))$ is a local Darboux frame of $\Gamma_g$, and such that $p(\text{Supp}(\Omega))\subset \text{Span}_{\mathbb{Z}}\{p(\gamma^i)\}$ (see for example \cite[Lemma 3.14]{QKBPS}). Extending $(\widetilde{\gamma}_i,\gamma^i)$ to a frame $(\widetilde{\gamma}_i,\gamma^i,\gamma_f^j)$ of $\Gamma$, the GMN equations for $\{\mathcal{X}_{\gamma^i}(\zeta),\mathcal{X}_{\widetilde{\gamma}_i}(\zeta),\mathcal{X}_{\gamma_f^j}(\zeta)\}$ reduce to the formulas
 
 \begin{equation}\label{mutloctwistcoord}
     \begin{split}
         \mathcal{X}_{\gamma^i}(x,\zeta)&=\mathcal{X}_{\gamma^i}^{\text{sf}}(x,\zeta)\, ,\\
         \mathcal{X}_{\gamma_f^j}(x,\zeta)&=\mathcal{X}_{\gamma_f^j}^{\text{sf}}(x,\zeta)\,, \\
         \mathcal{X}_{\widetilde{\gamma}_i}(x,\zeta)=\mathcal{X}_{\widetilde{\gamma}_i}^{\text{sf}}(x,\zeta)\exp \Big[-\frac{1}{4\pi i}\sum_{\gamma'\in \Gamma_{\pi(x)}}&\Omega(\gamma')\langle \widetilde{\gamma}_i,\gamma' \rangle\int_{\mathbb{R}_{-}Z_{\gamma'}}\frac{d\zeta'}{\zeta'}\frac{\zeta'+\zeta}{\zeta'-\zeta}\log(1-\mathcal{X}_{\gamma'}^{\text{sf}}(x,\zeta'))\Big]\,.\\
     \end{split}
 \end{equation}

One can then write
 
    \begin{equation}
    \varpi(\zeta)=\frac{1}{8\pi^2 }\langle \mathrm{d}\log(\mathcal{X}(\zeta))\wedge \mathrm{d}\log(\mathcal{X}(\zeta))\rangle=\frac{1}{4\pi^2}\mathrm{d}\log(\mathcal{X}_{\widetilde{\gamma}_i}(\zeta))\wedge \mathrm{d}\log(\mathcal{X}_{\gamma^i}^{\text{sf}}(\zeta))\,. 
    \end{equation}
    
 \begin{rem}\label{flavorcontribution}
 \leavevmode
 \begin{itemize}
     \item The sum in the formula for $\mathcal{X}_{\widetilde{\gamma}_i}(x,\zeta)$ converges normally on appropriate compact subsets of the parameters. Indeed, notice that by the support property (\ref{supportproperty}) we have $|Z_{\gamma'}|\to \infty$ as $|\gamma'|\to \infty$ with $\gamma'\in \text{Supp}(\Omega)$, so using that \begin{equation}
    |\mathcal{X}_{\gamma'}^{\text{sf}}(x,\zeta')|\leq \exp(-2\pi |Z_{\gamma'}|), \;\;\; \text{for} \;\;\; \zeta' \in \mathbb{R}_{-}Z_{\gamma'}\,,
    \end{equation}
     we can replace $|\log(1-\mathcal{X}_{\gamma}(x,\zeta'))|$ by $|\mathcal{X}_{\gamma'}(x,\zeta')|$ in estimating the integrals for $|\gamma'|$ big enough. Using the integral estimates from \cite[Section 3.4]{FGS} we then obtain

\begin{equation}
    \begin{split}
     \Big|\int_{\mathbb{R}_{-}Z_{\gamma'}}\frac{d\zeta'}{\zeta'}\frac{\zeta'+\zeta}{\zeta'-\zeta}\log(1-\mathcal{X}_{\gamma'}^{\text{sf}}(x,\zeta'))\Big|&\leq C\int_{0}^{\infty}ds \exp(-2\pi |Z_{\gamma'}|\cosh(s))\\
     &=CK_0(2\pi |Z_{\gamma'}|)\,,
     \end{split}
\end{equation}
where $C$ is a constant depending of the distance of $\zeta$ to $\mathbb{R}_{-}Z_{\gamma'}$, and $K_0$ is a modified Bessel function of the second kind. The convergence then follows from the  asymptotics $K_{0}(s)\sim \sqrt{\frac{2\pi}{s}}e^{-s}(1+\mathcal{O}(1/s))$ as $s\to \infty$ together with the convergence property of the BPS structures (\ref{convergenceproperty}). 
\item  An immediate consequence of (\ref{mutloctwistcoord}) is that if $\gamma_f \in \text{Supp}(\Omega)\cap \Gamma_f$, then $\Omega(\gamma_f)$ does not make any contribution to the twistor coordinates (since $\langle \widetilde{\gamma}_i, \gamma_f \rangle=0$), and hence $\Omega(\gamma_f)$ does not make any contribution to the HK structure described by $\varpi(\zeta)$. 
 \end{itemize}
 
 \end{rem}

The resulting candidate K\"{a}hler forms are then given by (see \cite[Section 4.3 and 5.6]{GMN1} and \cite[Lemma 3.10, Theorem 3.13]{QKBPS}):

\begin{equation}\label{holsymp}
    \omega_1+i\omega_2:=-\frac{1}{2\pi}\langle \mathrm{d}Z\wedge \mathrm{d}\theta \rangle + \sum_{\gamma}\Big(\Omega(\gamma)\mathrm{d}Z_{\gamma}\wedge A_{\gamma}^{\text{inst}}  +\frac{i\Omega(\gamma)}{2\pi }V_{\gamma}^{\text{inst}}\mathrm{d}\theta_{\gamma}\wedge \mathrm{d}Z_{\gamma}\Big)\,,
\end{equation}
\begin{equation}\label{kf}
    \omega_{3}:=\frac{1}{4}\langle \mathrm{d}Z\wedge \mathrm{d}\overline{Z}\rangle-\frac{1}{8\pi^2} \langle \mathrm{d}\theta\wedge \mathrm{d}\theta \rangle +\sum_{\gamma}\Big(\frac{i\Omega(\gamma)}{2}V^{\text{inst}}_{\gamma}\mathrm{d}Z_{\gamma}\wedge \mathrm{d}\overline{Z}_{\gamma}+\frac{\Omega(\gamma)}{2\pi } \mathrm{d}\theta_{\gamma}\wedge A_{\gamma}^{\text{inst}}\Big)\,,
\end{equation}
where 
\begin{equation}
\begin{split}
        V_{\gamma}^{\text{inst}}&:=\frac{1}{2\pi}\sum_{n>0}e^{in\theta_{\gamma}}K_0(2\pi n|Z_{\gamma}|)\,,\\
        A_{\gamma}^{\text{inst}}&:=-\frac{1}{4\pi}\sum_{n>0}e^{in\theta_{\gamma}}|Z_{\gamma}|K_1(2\pi n|Z_{\gamma}|)\Big( \frac{\mathrm{d}Z_{\gamma}}{Z_{\gamma}}-\frac{\mathrm{d}\overline{Z}_{\gamma}}{\overline{Z}_{\gamma}}\Big)\,,
        \end{split}
        \end{equation}
and $K_{\nu}$ for $\nu=0,1$ denote modified Bessel functions of the second kind.

\begin{rem}\leavevmode

\begin{itemize}
\item We remark that the formulas for the K\"{a}hler forms $\omega_{\alpha}$ are global, due to the monodromy invariance of $\Omega$. Furthermore, all the infinite sums converge normally over compact subsets of $\mathcal{N}(\theta_f)$. This is due to the strong convergence property of the variations of the BPS structures, together with the asymptotics of the Bessel functions $K_{\nu}(x)\sim \sqrt{\frac{2\pi}{x}}e^{-x}(1+\mathcal{O}(1/x))$ as $x\to \infty$. For a more detailed argument, see \cite[Lemma 3.9 and Lemma 3.10]{QKBPS}.
\item It is also worth noticing that while the expression of $\varpi(\zeta)$ in the mutually local case can be explicitly computed in terms of Bessel functions as above, the integral formulas in (\ref{mutloctwistcoord}) do not have an explicit form (as far as the authors know).
\end{itemize}
\end{rem}

In order to state precisely over which open subset $N\subset \mathcal{N}(\theta_f)$ the forms $\omega_{\alpha}$ actually define a (possibly indefinite) HK structure, we need the following definition (see also \cite[Definition 3.12]{QKBPS}).

\begin{dfn} Consider a mutually local HK variation of BPS structures $(M,\Gamma,Z,\Omega)$, and a fixed homomorphism $\theta_f:\Gamma_f \to \mathbb{R}/2\pi \mathbb{Z}$.  

\begin{itemize}
\item If $(M,g_{\text{sK}})$ denotes the associated ASK geometry to $(M,\Gamma,Z)$, then we define the following tensor field $T$ on the total space of $\pi:\mathcal{N}(\theta_f)\to M$:

\begin{equation}\label{nondegtensor}
    T:=\pi^*g_{\text{sK}}+ \sum_{\gamma}\Omega(\gamma)V_{\gamma}^{\text{inst}}\pi^*|\mathrm{d}Z_{\gamma}|^2\,.
\end{equation}
\item We denote by $N\subset \mathcal{N}(\theta_{f})$ the open subset where the tensor field $T$ is horizontally non-degenerate with respect to $\pi: \mathcal{N}(\theta_{f})\to M$ (i.e. it is non-degenerate on the normal bundle of the fibers of $\pi: \mathcal{N}(\theta_{f})\to M$).
\end{itemize}

\end{dfn}
\begin{rem}
The horizontal non-degeneracy condition on $T$ is equivalent to the non-degeneracy of $\omega_{\alpha}$ for $\alpha=1,2,3$ (see the proof of \cite[Theorem 3.13]{QKBPS}).
\end{rem}

Finally, in order to write down local explicit expressions for the HK metric, we will need the following:

\begin{dfn} \label{impexp} As before, we pick a local frame $(\widetilde{\gamma}_i,\gamma^i,\gamma_f^i)$ of $\Gamma$, such that $(p(\widetilde{\gamma}_i),p(\gamma^i))$ is a local Darboux frame of $\Gamma_g$, and $p(\text{Supp}(\Omega))\subset \text{Span}_{\mathbb{Z}}\{p(\gamma^i)\}$. Given $\gamma \in \text{Supp}(\Omega)$, we can then write

\begin{equation}
    \gamma=n_{i,g}(\gamma)\gamma^i + n_{i,f}(\gamma)\gamma^i_f
\end{equation}
for uniquely determined $n_{i,g}(\gamma),n_{i,f}(\gamma)\in \mathbb{Z}$.\\

With respect to this frame we define:

\begin{equation}
    W_i:=\mathrm{d}\theta_{\widetilde{\gamma}_i}-\tau_{ij}\mathrm{d}\theta_{\gamma^j}, \;\;\;\;\;\; W_i^{\text{inst}}:=\sum_{\gamma \in \text{Supp}(\Omega)}\Omega(\gamma)n_{i,g}(\gamma)(2\pi A_{\gamma}^{\text{inst}}-iV_{\gamma}^{\text{inst}}\mathrm{d}\theta_{\gamma})\,,
\end{equation}
\begin{equation}
        N_{ij}:=\text{Im}{\tau}_{ij}, \;\;\;\;\;\; N_{ij}^{\text{inst}}:=\sum_{\gamma}\Omega(\gamma)V^{\text{inst}}_{\gamma}n_{i,g}(\gamma)n_{j,g}(\gamma)\,,
\end{equation}
where $\mathrm{d}Z_{\widetilde{\gamma}_i}=\tau_{ij}\mathrm{d}Z_{\gamma^i}$

\end{dfn}
With the previous definitions at hand, we then have the following slight extension of \cite[Theorem 3.13, Lemma 3.15 and Corollary 3.17]{QKBPS} for the case where $\Gamma_f \neq 0$:

\begin{thm}\label{theorem1}
Let $(M,\Gamma,Z,\Omega)$ be a mutually local HK variation of BPS structures and fix a homomorphism $\theta_{f}:\Gamma_f \to \mathbb{R}/2\pi \mathbb{Z}$. Then the $\omega_{\alpha} \in \Omega^2(\mathcal{N}(\theta_f))$ given in (\ref{holsymp}) and (\ref{kf}) define a (possibly-indefinite) HK structure on $N\subset \mathcal{N}(\theta_{f})$. We denote this HK structure by $(N,g_N,\omega_1,\omega_2,\omega_3)$. With respect to the local frame $(\widetilde{\gamma}_i,\gamma^i,\gamma_f^i)$ of $\Gamma$ from Definition \ref{impexp} we have

\begin{equation}
    g_N=\mathrm{d}Z_{\gamma^i}(N_{ij}+N_{ij}^{inst})\mathrm{d}\overline{Z}_{\gamma^j}+\frac{1}{4\pi^2}(W_i+W_i^{inst})(N+N^{inst})^{ij}(\overline{W}_j+\overline{W}_j^{inst})\,.
\end{equation}
Furthermore, 
\begin{equation}\label{holsymw}
    \omega_1+i\omega_2=\frac{1}{2\pi}\mathrm{d}Z_{\gamma^i}\wedge (W_i+W_i^{\text{inst}})\,,
\end{equation}
and
\begin{equation}
    \omega_3=\frac{i}{2}(N_{ij}+N_{ij}^{\text{inst}})\mathrm{d}Z_{\gamma^i}\wedge \mathrm{d}\overline{Z}_{\gamma^j}+\frac{i}{8\pi^2}(N+N^{\text{inst}})^{ij}(W_i+W_i^{\text{inst}})\wedge (\overline{W}_j+\overline{W}_j^{\text{inst}})\,.
\end{equation}
\end{thm}

\begin{proof}
This is proven in Lemma $3.15$, Theorem $3.13$ and Corollary $3.17$ of  \cite{QKBPS} for the case $\Gamma_{f}=0$. The same proofs follow for $\Gamma_{f}\neq 0$, provided we use the expressions in Definition \ref{impexp} instead of those is \cite[Lemma 3.15]{QKBPS}.
\end{proof}

\subsection{The instanton corrected HK geometry associated to the resolved conifold}\label{instconifold}

We now restrict to the case of the resolved conifold. We  consider $(M,\Gamma, Z,\Omega)$ where

\begin{itemize}
    \item $(M,\Gamma,Z)$ is the central charge ASK geometry defined in section \ref{ASKgeometry}. 
    \item We define $\Omega: \Gamma \to \mathbb{Z}$ by the usual BPS spectrum associated to the resolved conifold, see e.~g.~\cite{BridgelandCon}:
    \begin{equation}
\Omega(\gamma) =
  \begin{cases}
    1       & \quad \text{if } \gamma = \pm \beta + n \delta\quad \text{for }\;\; n\in \mathbb{Z}\\
    -2  & \quad \text{if } \gamma= k \delta \quad \text{for} \;\;k\in \mathbb{Z}\setminus \{ 0\} \\
    0 & \quad \text{otherwise}\,.
  \end{cases}
\end{equation}
\end{itemize}

It is then easy to check that $(M,\Gamma, Z,\Omega)$ defined as above gives a mutually local HK variation of BPS structures, as in Definition \ref{HKBPS}.\\ 

By fixing a homomorphism $\theta_f:\Gamma_f \to \mathbb{R}/ 2 \pi \mathbb{Z}$, one then obtains an instanton corrected HK manifold $(N\subset \mathcal{N}(\theta_f),g_{N},\omega_1,\omega_2,\omega_3)$ of possibly indefinite signature. In this case, using Theorem \ref{theorem1} we find that the HK metric has an explicit local expression given in terms of the local frame $(\beta^{\vee},\beta,\delta)$ of $\Gamma$ by

\begin{equation}\label{conHKmetric}
    g=(N_{\beta}+N^{\text{inst}}_{\beta})|\mathrm{d}Z_{\beta}|^2+\frac{1}{4\pi^2}(N_{\beta}+N^{\text{inst}}_{\beta})^{-1}|W+W^{\text{inst}}|^2 \,,
\end{equation}
where 
\begin{equation}\label{instW}
    W=\mathrm{d}\theta_{\beta^{\vee}}-\tau \mathrm{d}\theta_{\beta}, \;\;\;\;\;\; W^{\text{inst}}=\sum_{\gamma \in \{\pm \beta +\mathbb{Z}\delta\}}\Omega(\gamma)n_g(\gamma)(2\pi A_{\gamma}^{\text{inst}}-iV_{\gamma}^{\text{inst}}\mathrm{d}\theta_{\gamma})\,,
\end{equation}
\begin{equation}
        N_{\beta}:=\text{Im}(\tau), \;\;\;\;\;\; N^{\text{inst}}_{\beta}:=\sum_{\gamma \in \{\pm\beta +\mathbb{Z}\delta\}}\Omega(\gamma)V^{\text{inst}}_{\gamma}(n_g(\gamma))^2\,.
\end{equation}

\begin{rem}\label{flavorcontributionconifold}
As we previously mentioned in Remark \ref{flavorcontribution}, since the charge $\delta$ is a flavor charge in our setting,  the BPS indices $\Omega(k\delta)=-2$ do not contribute in the instanton corrections of the HK metric.  Hence, the HK geometry we consider only captures the part corresponding to $\Omega(\beta+n\delta)=1$.
\end{rem}

 In particular, when we pick $\theta_f=0$ (i.e. $\theta_{\delta}=0)$ we can say the following about $N\subset \mathcal{N}(\theta_f=0)$:

\begin{prop}\label{Ncar} We denote as before by $M_+\subset M$ the connected component of $M$ containing a neighborhood of $t=0$, where the ASK geometry is positive definite. Then $N\subset \mathcal{N}(\theta_f=0)$ consists of an open subset of $\mathcal{N}(\theta_f=0)$ containing the subset $N_0:=\{\theta\in \mathcal{N}(\theta_f)\; |\; \pi(\theta)\in M_{+},\;\theta_{\beta}=0\}$.

\end{prop}

\begin{proof}
This follows immediately by noticing that the tensor $T$ from \eqref{nondegtensor} in this case has the form (we omit pullbacks by $\pi$ from the notation) 

\begin{equation}\label{Ttensorconifold}
    T=\left(\text{Im}(\tau) +\frac{1}{\pi}\sum_{n\in \mathbb{Z}}\sum_{m>0}\cos(m\theta_{\beta})K_0(2\pi|m||t-n|)\right) |dt|^2
\end{equation}
In particular, we have 

\begin{equation}
    T|_{N_0}=\left(\text{Im}(\tau) +\frac{1}{\pi}\sum_{n\in \mathbb{Z}}\sum_{m>0}K_0(2\pi|m||t-n|)\right) |dt|^2
\end{equation} which is horizontally non-degenerate, since $\text{Im}(\tau)>0$ on $M_+$ and $K_0(x)>0$ for $x>0$.
\end{proof}

\begin{rem}\label{remN} Notice that while $N$ contains points with arbitrary $\theta_{\beta^{\vee}}$ value (since $T$ is independent of $\theta_{\beta^{\vee}}$), it is a priori not clear that $T$ remains horizontally non-degenerate for all values of $\theta_{\beta}$. On the other hand, in the construction of instanton corrected HK metrics from \cite{GMN1}, there is an arbitrary free parameter $R>0$ that we have set to $R=1$ in this paper. One can reintroduce this parameter by scaling the central charge $Z\to RZ$ and the twistor family of holomorphic symplectic forms $\varpi(\zeta)\to \frac{1}{R}\varpi(\zeta)$.  Once this is done, it is easy to check that due to exponential decay of the $K_0(x)$ as $x\to \infty$, by taking $R$ very big and $t$ in a small neighborhood of $0$, we can make $\text{Im}(\tau)+ \sum_{m>0}\cos(m\theta_{\beta})K_0(2\pi R|m||t|)$ dominate over the other terms of the sum $\frac{1}{\pi}\sum_{n\in \mathbb{Z}-\{0\}}\sum_{m>0}K_0(2\pi R|m||t-n|)$ of \eqref{Ttensorconifold}. Furthermore, $\left(\text{Im}(\tau)+ \sum_{m>0}\cos(m\theta_{\beta})K_0(2\pi R|m||t|)\right)|dt|^2$ behaves like what one has for the Ooguri-Vafa space (see Section \ref{ovsection} below), which does remain non-degenerate for all values of $\theta_{\beta}$ and $t$ sufficiently small \cite{LCS}. Hence, it seems that by reintroducing a sufficiently big $R$ factor, one can guarantee that $N$ contains a torus fibration over a punctured neighborhood of $t=0$. 

\end{rem}

Using the notation:
\begin{equation}\label{notationinst}
    \begin{split}
    V_{n}^{\text{inst}}&:=V_{\beta-n\delta}^{\text{inst}}+V_{-\beta+n\delta}^{\text{inst}}=\frac{1}{2\pi}\sum_{m \in \mathbb{Z}-\{0\}}e^{im\theta_{\beta-n\delta}}K_0(2\pi |m||t-n|)\,,\\
    A_{n}^{\text{inst}}&:=A_{\beta-n\delta}^{\text{inst}}-A_{-\beta+n\delta}^{\text{inst}}\\
    &=-\frac{1}{4\pi}\Big(\sum_{m \in \mathbb{Z}-\{0\}}\text{sign}(m)e^{im\theta_{\beta-n\delta}}|t-n|K_1(2\pi |m||t-n|)\Big)\Big( \frac{\mathrm{d}t}{t-n}-\frac{\mathrm{d}\overline{t}}{\overline{t}-n}\Big)\,,\\
    \end{split}
\end{equation}
and the explicit values of $\Omega(\gamma)$, we can then rewrite the instanton correction terms as

\begin{equation}
    \begin{split}
    N^{\text{inst}}_{\beta}&:=\sum_{n \in \mathbb{Z}}V^{\text{inst}}_{n}\,,\\
    W^{\text{inst}}&:=\sum_{n \in \mathbb{Z}}\Big(2\pi A_{n}^{\text{inst}}-iV_{n}^{\text{inst}}\mathrm{d}\theta_{\beta}\Big)\,.
    \end{split}
\end{equation}

Hence, using Theorem \ref{theorem1} we find

\begin{cor}\label{HKexplicit}
The instanton corrected HK metric $(N,g_N,\omega_1,\omega_2,\omega_3)$ associated to the resolved conifold has the following explicit form:

\begin{equation}
    \begin{split}
    g_N&=\Big(-\frac{1}{2\pi}\log\Big|\frac{1-e^{2\pi it}}{e^{2\pi it}}\Big|+ \sum_{n \in \mathbb{Z}}V_{n}^{\text{inst}}\Big)|\mathrm{d}t|^2\\
    &+\frac{1}{4\pi^2}\Big(-\frac{1}{2\pi}\log\Big|\frac{1-e^{2\pi it}}{e^{2\pi it}}\Big|+ \sum_{n \in \mathbb{Z}}V_{n}^{\text{inst}}\Big)^{-1}\\
    &\;\;\;\;\;\;\cdot\Big|\mathrm{d}\theta_{\beta^{\vee}}-\frac{1}{2\pi i}\Big(\log(1-e^{2\pi it})-\log(e^{2\pi it})\Big)\mathrm{d}\theta_{\beta}+\sum_{n \in \mathbb{Z}}\Big(2\pi A_{n}^{\text{inst}}-iV_{n}^{\text{inst}}\mathrm{d}\theta_{\beta}\Big)\Big|^2 \,.
    \end{split}
\end{equation}
\end{cor}

We finish this subsection with a description of a certain Lagrangian submanifold $L$ of the holomorphic symplectic manifold $(N,I_3,\omega_1+i\omega_2)$. This Lagrangian submanifold will play a role in the conformal limit of Section \ref{confRH}.

\begin{prop} \label{Lagrangian}Fix $\theta_f=0$ (and hence $\theta_{\delta}=0$). Then the submanifold $L\subset N \subset \mathcal{N}(\theta_f=0)$ defined by setting $\theta_{\beta}=\theta_{\beta^{\vee}}=0$ is a complex Lagrangian submanifold of the holomorphic symplectic manifold $(N,I_3,\omega_1+i\omega_2)$. Furthermore, $L$ has a component $L_{+}$ biholomorphic to  $M_+$.

\end{prop}
\begin{proof}
We pick the local frame $(\beta^{\vee},\beta,\delta)$ of $\Gamma$. With respect to this frame, one has by Theorem \ref{theorem1} the following expression

\begin{equation}
    \omega_1+i\omega_2=\frac{1}{2\pi}\mathrm{d}Z_{\beta}\wedge (W+W^{\text{inst}})\,,
\end{equation}
and $(\mathrm{d}Z_{\beta}$, $W+W^{\text{inst}})$ gives a local frame of $(1,0)$ forms with respect to complex structure $I_3$ (see for example the proof \cite[Theorem 3.13]{QKBPS} for more details). Now notice that
\begin{equation}
A_{n}^{\text{inst}}|_L=-\frac{1}{4\pi}\Big(\sum_{m \in \mathbb{Z}-\{0\}}\text{sign}(m)|t-n|K_1(2\pi |m||t-n|)\Big)\Big( \frac{\mathrm{d}t}{t-n}-\frac{\mathrm{d}\overline{t}}{\overline{t}-n}\Big)=0 \,.
\end{equation}
This implies that
\begin{equation}\label{Wvanish}
    (W+W^{\text{inst}})|_L=\Big(\mathrm{d}\theta_{\beta^{\vee}}-\tau \mathrm{d}\theta_{\beta}+ \sum_{n \in \mathbb{Z}}\Big(2\pi A_{n}^{\text{inst}}-iV_{n}^{\text{inst}}\mathrm{d}\theta_{\beta}\Big)\Big)\Big|_L=0 \,.
\end{equation}

We then conclude that $L$ is a Lagrangian submanifold of the holomorphic symplectic manifold $(N,I_3,\omega_1+\omega_2)$.\\

For the second statement, it suffices to notice that $t=Z_{\beta}$ is a global holomorphic coordinate for $L$, and that  $L_{+}:=\pi^{-1}(M_{+})\cap L\subset N_0\subset N$ by Proposition \ref{Ncar}.
\end{proof}

\begin{cor} With the same hypotheses as Proposition \ref{Lagrangian}, if $N_{+}\subset N$ is a connected open subset of $N$ such that $L_{+}\subset N_{+}$, the $g_{N}$ is positive definite on $N_{+}$. 

\end{cor}

\begin{proof}

Writing (\ref{conHKmetric}) in the local real frame given by $\text{Re}(dt)$, $\text{Im}(dt)$, $\text{Re}(W+W^{\text{inst}})$, $\text{Im}(W+W^{\text{inst}})$, one sees that the signature of $g_{N}$ is determined the sign of $N_{\beta}+N^{\text{inst}}_{\beta}$. Indeed, if $N_{\beta}+N^{\text{inst}}_{\beta}>0$ (resp. $N_{\beta}+N^{\text{inst}}_{\beta}<0$) we find that $g_{N}$ is positive definite (resp. negative definite).  \\

Now notice that

\begin{equation}
    V_{n}^{\text{inst}}|_L=\frac{1}{2\pi}\sum_{m \in \mathbb{Z}-\{0\}}K_0(2\pi |m||t-n|)>0\,,
\end{equation}
since $K_0(x)>0$ for $x>0$. In particular, since $\text{Im}(\tau)>0$ on $M_{+}$, one finds that 

\begin{equation}
    (N_{\beta}+N^{\text{inst}}_{\beta})|_{L_{+}}=\Big(\text{Im}(\tau)+\sum_{n \in \mathbb{Z}}V_{n}^{\text{inst}}\Big)\Big|_{L_{+}}>0\,,
\end{equation}
where the pullback of $\text{Im}(\tau)$ to $N$ is suppressed. Hence, $g_{N}$ is positive definite on $N_{+}\supset L_{+}$.

\end{proof}

\subsection{Ooguri-Vafa-like smoothing}\label{ovsection}

In this section, we study the instanton corrected HK manifold $(N,g_{N},\omega_1,\omega_2,\omega_3)$ associated to the resolved conifold near the locus $t=0$. We show that the HK structure $(N,g_{N},\omega_1,\omega_2,\omega_3)$ admits a smooth extension over the locus $t=0$, improving on the semi-flat HK structure built out of the ASK geometry associated to the resolved conifold.  This result is an instance of a more general conjecture stated in \cite[Section 7]{NewCHK}, where certain instanton corrected HK manifolds are expected to admit smooth extensions over certain singular loci where the semi-flat HK metric is not defined. \\

In order to show this, we will need to first describe the Ooguri-Vafa metric.  This metric can be described via the following HK variation of BPS structures $(M^{\text{ov}},\Gamma^{\text{ov}},Z^{\text{ov}},\Omega^{\text{ov}})$ (see \cite{NewCHK}, \cite{LCS} for another treatment, and \cite{OV} for the origin of this metric):

\begin{itemize}
    \item We let $M^{\text{ov}}\subset \mathbb{C}^{\times}$ be a small neighborhood of $t=0$ (without $0$). We assume for simplicity that $M^{\text{ov}}\subset M_{+}$, where $M_{+}$ is the component where the ASK geometry associated to the resolved conifold is positive definite (see Section \ref{ASKgeometry}). 
    \item We take $\Gamma^{\text{ov}}:=\Gamma_g|_{M^{\text{ov}}}$, where $\Gamma_g$ was defined for the central charge ASK geometry associated to the resolved conifold in Section \ref{ASKgeometry}. In particular, we have $\Gamma^{\text{ov}}_f=0$.
    \item With respect to the local Darboux frame $(\beta^{\vee},\beta)$ of $\Gamma_g$, we first define $Z^{\text{ov}}$ as follows: 
    \begin{equation}
        \begin{split}
        Z_{\beta}^{\text{ov}}&:=t\,,\\
        Z_{\beta^{\vee}}^{\text{ov}}&:=\frac{1}{2\pi i}(t\log(t/\Lambda)-t)\,,
        \end{split}
    \end{equation}
    for some choice of $\Lambda \in \mathbb{C}^{\times}$. Then as we go around $t=0$ we have 
    \begin{equation}
    Z_{\beta^{\vee}}^{\text{ov}}\to Z_{\beta^{\vee}}^{\text{ov}}+Z_{\beta}^{\text{ov}}\,.
    \end{equation}
    Since this matches the jump $\beta^{\vee} \to \beta^{\vee}+\beta$ in $\Gamma^{\text{ov}}$, $Z^{\text{ov}}$ defines a holomorphic  section of $(\Gamma^{\text{ov}})^{*}\otimes \mathbb{C}\to M^{\text{ov}}$.
    \item Finally, we set $\Omega^{\text{ov}}(\pm \beta)=1$ and $0$ otherwise. 
\end{itemize}
It is easy to check that $(M^{\text{ov}},\Gamma^{\text{ov}},Z^{\text{ov}})$ is a central charge ASK geometry for $M^{\text{ov}}$ sufficiently small, and that $(M^{\text{ov}},\Gamma^{\text{ov}},Z^{\text{ov}},\Omega^{\text{ov}})$ is an HK variation of mutually local BPS structures. 

\begin{dfn}
The HK manifold $(N^{\text{ov}},g^{\text{ov}},\omega_1^{\text{ov}},\omega_2^{\text{ov}},\omega_3^{\text{ov}})$ associated to the mutually local HK variation of BPS structures $(M^{\text{ov}},\Gamma^{\text{ov}},Z^{\text{ov}},\Omega^{\text{ov}})$ via Theorem \ref{theorem1} is called the Ooguri-Vafa space. The definition of the Ooguri-Vafa space depends on the choice of $\Lambda \in \mathbb{C}^{\times}$, which we will call the cut-off parameter.
\end{dfn}

If we denote by $\mathcal{N}^{\text{ov}}$ the total space of the corresponding torus bundle $\mathcal{N}^{\text{ov}}\to M^{\text{ov}}$, then the Ooguri-Vafa HK structure turns out to the defined on the whole of $\mathcal{N}^{\text{ov}}$ (see \cite{GMN1,LCS}), so that $N^{\text{ov}}= \mathcal{N}^{\text{ov}}\cong \mathcal{N}(\theta_f=0)|_{M^{\text{ov}}}$. Furthermore, $g^{\text{ov}}$ turns out to be positive definite, and the hyperk\"{a}hler structure of the Ooguri-Vafa space extends smoothly over the locus $t=0$ (see \cite[Proposition 3.2]{LCS}). In terms of the torus bundle $\mathcal{N}^{\text{ov}} \to M^{\text{ov}}$, one adds a singular fiber of Kodaira type $I_1$ (a pinched torus) over $t=0$. \\ 

For reference, it will be useful to collect some formulas for the K\"{a}hler forms $\omega_{\alpha}^{\text{ov}}$.

\begin{prop}
The forms $\varpi^{\text{ov}}:=\omega_1^{\text{ov}}+i\omega_2^{\text{ov}}$ and $\omega_3^{\text{ov}}$ have the following expressions with respect to the coordinates $(t,\theta_{\beta},\theta_{\beta^{\vee}})$:

\begin{equation}\label{ovforms}
    \begin{split}
        \varpi^{\text{ov}}&=\frac{1}{2\pi}\mathrm{d}t \wedge Y^{\text{ov}}\,,\\
        \omega_3^{\text{ov}}&=\frac{1}{4}\langle \mathrm{d}Z^{\text{ov}}\wedge \mathrm{d}\overline{Z}^{\text{ov}}\rangle-\frac{1}{8\pi^2} \langle \mathrm{d}\theta\wedge \mathrm{d}\theta \rangle +\frac{i}{2}V^{\text{inst}}_{0}\mathrm{d}t\wedge \mathrm{d}\overline{t}+\frac{1}{2\pi } \mathrm{d}\theta_{\beta}\wedge A_{0}^{\text{inst}}\,,\\
    \end{split}
\end{equation}
where 
\begin{equation}\label{ovY}
    Y^{\text{ov}}=\mathrm{d}\theta_{\beta^{\vee}}-\tau^{\text{ov}}\mathrm{d}\theta_{\beta}+2\pi A_{0}^{\text{inst}}-iV_{0}^{\text{inst}}\mathrm{d}\theta_{\beta}, \;\;\;\; \tau^{\text{ov}}:=\frac{\mathrm{d}Z_{\beta^{\vee}}^{\text{ov}}}{\mathrm{d}Z_{\beta}^{\text{ov}}}=\frac{1}{2\pi i}\log\Big(\frac{t}{\Lambda}\Big)\,.
\end{equation}
and $A_0^{\text{inst}}$, $V_0^{\text{inst}}$ are as in (\ref{notationinst}) for $n=0$.
\end{prop}

\begin{proof}
The formulas for $\varpi^{\text{ov}}$ and $\omega_3^{\text{ov}}$ follow immediately from equations (\ref{holsymw}) and (\ref{kf}), after specializing to the variation of BPS structures of the Ooguri-Vafa space, together with the notation (\ref{notationinst}).
\end{proof}

\begin{thm} \label{ovsmoothing} Consider the HK manifold $(N,g_{N},\omega_1,\omega_2,\omega_3)$ associated to the resolved conifold, and the Ooguri-Vafa HK manifold $(N^{\text{ov}},g^{\text{ov}},\omega_1^{\text{ov}},\omega_2^{\text{ov}},\omega_3^{\text{ov}})$ with cutoff $\Lambda=\frac{i}{2\pi}$.\\

We then have

\begin{equation}
    \omega_{\alpha}=\omega_{\alpha}^{\text{ov}}+\eta_{\alpha} \;\;\;\; \text{for} \;\;\;\; \alpha=1,2,3,
\end{equation}
where $\eta_{\alpha}$ are two forms extending over the locus $t=0$, satisfying 

\begin{equation}\label{exteta}
\begin{split}
    \eta_1|_{t=0}&=0\,,\\
    \eta_2|_{t=0}&=-\frac{1}{2\pi}\Big(\sum_{n \in \mathbb{Z}, n\neq 0}V_n^{\text{inst}}|_{t=0}\Big)\mathrm{d}t \wedge \mathrm{d}\theta_{\beta}\,,\\
    \eta_3|_{t=0}&=\frac{i}{2}\Big(\sum_{n \in \mathbb{Z}, n\neq 0}V_n^{\text{inst}}|_{t=0}\Big)\mathrm{d}t\wedge \mathrm{d}\overline{t}\,.\\
\end{split}   
\end{equation}
In particular, the HK structure of $(N,g_{N},\omega_1,\omega_2,\omega_3)$ extends smoothly over the points in the locus $t=0$ where the forms $\omega_{\alpha}$ remain non-degenerate.

\end{thm}
\begin{rem}
 Clearly $dt$ extends over $t=0$. On the other hand, since $\beta$ is a global section of $\Gamma$, $d\theta_{\beta}$ also extends over $t=0$. Hence, the expressions on (\ref{exteta}) make sense. 
\end{rem}
\begin{proof}
In the following, we make frequent use of the following identity near $t=0$
\begin{equation}\label{tau-tauov}
    \tau=\frac{1}{2\pi i}\log\Big(\frac{t}{\Lambda}\Big)+\mathcal{O}(t)=\tau^{\text{ov}}+\mathcal{O}(t), \;\;\;\; \Lambda=\frac{i}{2\pi}\,,
\end{equation}
where in the second equality we just used the definition of $\tau^{\text{ov}}$. \\

Now let $\varpi=\omega_1+i\omega_2$ and $\varpi^{\text{ov}}=\omega_1^{\text{ov}}+i\omega_2^{\text{ov}}$. Using (\ref{tau-tauov}), (\ref{ovforms}) and (\ref{ovY}) one finds

\begin{equation}
    \begin{split}
    \varpi&=\frac{1}{2\pi}\mathrm{d}t\wedge (W_i+W_i^{\text{inst}})\\
    &=\frac{1}{2\pi}\mathrm{d}t\wedge \Big(\mathrm{d}\theta_{\beta^{\vee}}-\tau \mathrm{d}\theta_{\beta} +\sum_{n \in \mathbb{Z}}\Big(2\pi A_n^{\text{inst}} -iV_{n}^{\text{inst}}\mathrm{d}\theta_{\beta}\Big)\Big)\\
    &=\frac{1}{2\pi}\mathrm{d}t\wedge \Big(Y^{\text{ov}} +\mathcal{O}(t)\mathrm{d}\theta_{\beta} + \sum_{n \in \mathbb{Z}, n\neq 0} 2\pi A_n^{\text{inst}} -i\sum_{n\in \mathbb{Z},n\neq 0}V_n^{\text{inst}}\mathrm{d}\theta_{\beta}\Big)\\
    &=\varpi^{\text{ov}}+\frac{1}{2\pi}\mathrm{d}t\wedge \Big(\mathcal{O}(t)\mathrm{d}\theta_{\beta} + \sum_{n \in \mathbb{Z}, n\neq 0} 2\pi A_n^{\text{inst}} -i\sum_{n\in \mathbb{Z},n\neq 0}V_n^{\text{inst}}\mathrm{d}\theta_{\beta}\Big)\,.
    \end{split}
\end{equation}
Now notice that
\begin{equation}
    \sum_{n \in \mathbb{Z}, n\neq 0} 2\pi A_n^{\text{inst}}|_{t=0}=\frac{1}{2}\sum_{n \in \mathbb{Z}, n\neq 0}\text{sign}(n)\Big(\sum_{m \in \mathbb{Z}-\{0\}}\text{sign}(m)e^{im\theta_{\beta-n\delta}}K_1(2\pi |m||n|)\Big)\Big( \mathrm{d}t-\mathrm{d}\overline{t}\Big)=0\,,\\
\end{equation}
since $A_{n}^{\text{inst}}|_{t=0}$ gets canceled by $A_{-n}^{\text{inst}}|_{t=0}$. Furthermore, since $\mathrm{d}t$ and $\mathrm{d}\theta_{\beta}$ extend over $t=0$ (since $\beta$ is a global non-vanishing section of $\Gamma$), we find that the tensor
\begin{equation}
    \eta:=\frac{1}{2\pi}\mathrm{d}t\wedge \Big(\mathcal{O}(t)\mathrm{d}\theta_{\beta} + \sum_{n \in \mathbb{Z}, n\neq 0} 2\pi A_n^{\text{inst}} -i\sum_{n\in \mathbb{Z},n\neq 0}V_n^{\text{inst}}\mathrm{d}\theta_{\beta}\Big)
\end{equation}
extends over $t=0$ and satisfies that

\begin{equation}
    \eta|_{t=0}=-\frac{i}{2\pi}\Big(\sum_{n\in \mathbb{Z},n\neq 0}V_n^{\text{inst}}|_{t=0}\Big)\mathrm{d}t\wedge \mathrm{d}\theta_{\beta}\,.
\end{equation} Taking $\eta_1:=\text{Re}(\eta)$ and $\eta_2:=\text{Im}(\eta)$, and using that $\sum_{n\in \mathbb{Z},n\neq 0}V_n^{\text{inst}}$ is real valued, we obtain the corresponding expressions (\ref{exteta}) for $\alpha=1,2$.\\

For $\omega_3$, we use (\ref{kf}) specialized to the case of the resolved conifold:

\begin{equation}\label{kfcon}
    \omega_3=\frac{1}{4}\langle \mathrm{d}Z\wedge \mathrm{d}\overline{Z}\rangle-\frac{1}{8\pi^2} \langle \mathrm{d}\theta\wedge \mathrm{d}\theta \rangle+\sum_{n \in \mathbb{Z}}\Big(\frac{i}{2}V^{\text{inst}}_{n}\mathrm{d}t\wedge \mathrm{d}\overline{t}+\frac{1}{2\pi } \mathrm{d}\theta_{\beta}\wedge A_{n}^{\text{inst}}\Big)\,.
\end{equation}

First notice that
\begin{equation}
    \begin{split}
    \frac{1}{4}\langle \mathrm{d}Z\wedge \mathrm{d}\overline{Z} \rangle &= \frac{i}{2}\text{Im}(\tau)\mathrm{d}t\wedge \mathrm{d}\overline{t}\\
    &=\frac{i}{2}\Big(\text{Im}(\tau^{\text{ov}})+\mathcal{O}(t)\Big)\mathrm{d}t\wedge \mathrm{d}\overline{t}\\
    &=\frac{1}{4}\langle \mathrm{d}Z^{\text{ov}}\wedge \mathrm{d}\overline{Z}^{\text{ov}} \rangle + \mathcal{O}(t)\mathrm{d}t\wedge \mathrm{d}\overline{t}\,,
    \end{split}
\end{equation}
so we can rewrite (\ref{kfcon}) as follows:
\begin{equation}
    \begin{split}
    \omega_3=\omega_3^{\text{ov}}+\mathcal{O}(t)\mathrm{d}t\wedge \mathrm{d}\overline{t} + \frac{i}{2}\sum_{n \in \mathbb{Z},n\neq 0}V^{\text{inst}}_{n}\mathrm{d}t\wedge \mathrm{d}\overline{t}+\frac{1}{2\pi } \mathrm{d}\theta_{\beta}\wedge \sum_{n \in \mathbb{Z},n\neq 0}A_{n}^{\text{inst}}\,.
    \end{split}
\end{equation}
Letting
\begin{equation}
    \eta_3:=\mathcal{O}(t)\mathrm{d}t\wedge \mathrm{d}\overline{t} + \frac{i}{2}\sum_{n \in \mathbb{Z},n\neq 0}V^{\text{inst}}_{n}\mathrm{d}t\wedge \mathrm{d}\overline{t}+\frac{1}{2\pi } \mathrm{d}\theta_{\beta}\wedge \sum_{n \in \mathbb{Z},n\neq 0}A_{n}^{\text{inst}}\,,
\end{equation}
it follows that $\eta_3$ extends over $t=0$ and

\begin{equation}
    \eta_3|_{t=0}=\frac{i}{2}\sum_{n \in \mathbb{Z},n\neq 0}V^{\text{inst}}_{n}|_{t=0}\mathrm{d}t\wedge \mathrm{d}\overline{t}\,.
\end{equation}

Since the triple $\omega^{\text{ov}}_{\alpha}$ extends smoothly over $t=0$ (see \cite[Proposition 3.2]{LCS}), we then conclude that the same holds for $\omega_{\alpha}$. Since the triple of K\"{a}hler forms $\omega_{\alpha}$ of a hyperk\"{a}hler structure determines the complex structures $I_{\alpha}$ and metric $g_N$, we conclude that the full HK structure $(N,g_{N},\omega_1,\omega_2,\omega_3)$ extends smoothly over the points in the locus $t=0$ where the forms $\omega_{\alpha}$ are non-degenerate. 
\end{proof}

\section{Conformal limit of the twistor coordinates and the quantum dilogarithm}\label{confRH}
In this section we study a certain conformal limit  of the twistor coordinates $\mathcal{X}_{\beta^{\vee}}$ and $\mathcal{X}_{\beta}$. We show that the conformal limit of $\mathcal{X}_{\beta^{\vee}}$ satisfies properties related to a solution of a Riemann-Hilbert problem considered in \cite{BridgelandCon}.  This will allow us to conjecture that $\mathcal{X}_{\beta^{\vee}}$ matches the corresponding unique solution to the Riemann-Hilbert problem.

\begin{rem}\label{rescalerem}
In order to be able to easily compare with the results of \cite{BridgelandCon}, it will be convenient to scale our initial choice of central charge for the resolved conifold by

\begin{equation}
    Z\to 2i Z=Z'\,,
\end{equation} 
and consider the modified HK variation of BPS structures $(M,Z',\Gamma,\Omega)$ associated to the resolved conifold. This modification will remove several factors of $2i$ from the final results. The corresponding twistor coordinates associated to $(M,Z',\Gamma,\Omega)$ are then given (in terms of $Z$) by

\begin{equation}
    \begin{split}
    \mathcal{X}_{\gamma}(x,\zeta)&=\exp[2\pi i \zeta^{-1}Z_{\gamma}+i\theta_{\gamma}-2\pi i  \zeta \overline{Z}_{\gamma}]\\
        &\;\;\;\;\; \cdot\exp \Big[-\frac{1}{4\pi i}\sum_{\gamma'\in \Gamma_{\pi(\theta)}}\Omega(\gamma')\langle \gamma,\gamma' \rangle\int_{\mathbb{R}_{-}2iZ_{\gamma'}}\frac{d\zeta'}{\zeta'}\frac{\zeta'+\zeta}{\zeta'-\zeta}\log(1-\mathcal{X}_{\gamma'}(x,\zeta'))\Big]\,.\\
    \end{split}
\end{equation}
\end{rem}

In the case of the resolved conifold, the instanton corrected HK geometry considered in Section \ref{instconifold} is described by the twistor coordinates $\mathcal{X}_{\beta}^{\text{sf}}(\zeta)$ and $\mathcal{X}_{\beta^{\vee}}(\zeta)$. This is the case since the $\mathcal{O}(2)$-family of holomorphic symplectic forms is given by

\begin{equation}
    \zeta\varpi(\zeta)\otimes \partial_{\zeta}=\frac{\zeta}{4\pi^2}\mathrm{d}\log( \mathcal{X}_{\beta^{\vee}}(\zeta))\wedge \mathrm{d}\log (\mathcal{X}_{\beta}^{\text{sf}}(\zeta)) \otimes \partial_{\zeta} \,, 
\end{equation}
where (using the rescaled central charge from Remark \ref{rescalerem})

\begin{equation}
    \begin{split}
        \mathcal{X}_{\beta}^{\text{sf}}(x,\zeta)&=\exp[2\pi i \zeta^{-1}Z_{\beta}+i\theta_{\beta}-2\pi i  \zeta \overline{Z}_{\beta}]\,,\\
        \mathcal{X}_{\beta^{\vee}}(x,\zeta)&=\exp[2\pi i \zeta^{-1}Z_{\beta^{\vee}}+i\theta_{\beta^{\vee}}-2\pi i  \zeta \overline{Z}_{\beta^{\vee}}]\\
        &\;\;\;\;\; \cdot\exp \Big[-\frac{1}{4\pi i}\sum_{\gamma'\in \Gamma_{\pi(x)}}\Omega(\gamma')\langle \beta^{\vee},\gamma' \rangle\int_{\mathbb{R}_{-}2iZ_{\gamma'}}\frac{d\zeta'}{\zeta'}\frac{\zeta'+\zeta}{\zeta'-\zeta}\log(1-\mathcal{X}_{\gamma'}^{\text{sf}}(x,\zeta'))\Big]\,.\\
    \end{split}
\end{equation}

Denoting $\mathcal{X}_{\beta^{\vee}}^{\text{inst}}:=\mathcal{X}_{\beta^{\vee}}/\mathcal{X}_{\beta^{\vee}}^{\text{sf}}$ and using the explicit spectrum of the resolved conifold, one can write
\begin{equation}
    \begin{split}
    \mathcal{X}_{\beta^{\vee}}^{\text{inst}}(\zeta)=&\exp\Big(-\frac{1}{4\pi i}\sum_{n \in \mathbb{Z}} \int_{\mathbb{R}_{-}2iZ_{\beta +n\delta}}\frac{d\zeta'}{\zeta'}\frac{\zeta'+\zeta}{\zeta'-\zeta}\log (1- \mathcal{X}_{\beta+n\delta}^{\text{sf}}(\zeta'))\\
    &+\frac{1}{4\pi i}\sum_{n \in \mathbb{Z}} \int_{\mathbb{R}_{-}2iZ_{-\beta +n\delta}}\frac{d\zeta'}{\zeta'}\frac{\zeta'+\zeta}{\zeta'-\zeta}\log (1- \mathcal{X}_{-\beta+n\delta}^{\text{sf}}(\zeta'))\Big) \,.
    \end{split}
\end{equation}

Even though one can explicitly compute $\mathrm{d}\log \mathcal{X}_{\beta^{\vee}}^{\text{inst}}(\zeta)\wedge \mathrm{d} \log \mathcal{X}_{\beta}^{\text{sf}}(\zeta)$ in terms of Bessel functions (see equations (\ref{holsymp}) and (\ref{kf}), and the computations in \cite[Section 4.3 and 5.6]{GMN1}), $\mathcal{X}_{\beta^{\vee}}^{\text{inst}}(\zeta)$ has no known explicit expression as far as the authors know.\\

In the following we will take a certain conformal limit, considered in \cite{Gaiotto:2014bza}, of the twistor coordinates $\mathcal{X}_{\gamma}(x,\zeta)$. We will make a slight abuse of notation and denote the expressions obtained after the conformal limit by $\mathcal{X}_{\gamma}(t,\lambda)$, where $t \in M$ and $\lambda \in \mathbb{C}^{\times}$. It will turn out that the expression one gets by taking this limit for $\mathcal{X}_{\beta^{\vee}}^{\text{inst}}(x,\zeta)$ does not converge absolutely, but only conditionally. Regarding the conditionally convergent expression for $\mathcal{X}_{\beta^{\vee}}^{\text{inst}}(t,\lambda)$:

\begin{itemize}
    \item We will specify in Proposition \ref{coordconv} how to sum the terms of  $\mathcal{X}_{\beta^{\vee}}^{\text{inst}}(t,\lambda)$ in order to obtain a convergent expression. We will then show that the convergent expression satisfies properties related (but not entirely equivalent) to the Riemann-Hilbert problem considered in \cite{BridgelandCon}. More precisely, we show that it satisfies property (RH1) from Section \ref{RHprob}, but we only manage to show that (RH2) and (RH3) are satisfied along certain sectors, instead of half-planes.
    \item  Since the Riemann-Hilbert admits a (unique) solution in terms of the Faddeev quantum dialogarithm (see \cite{BridgelandCon}), the previous results allow us to conjecture a relation between $\mathcal{X}_{\beta^{\vee}}^{\text{inst}}(t,\lambda)$ and the quantum dilogarithm.
\end{itemize}

In Section \ref{conflimsec} we start by quickly reviewing the conformal limit considered in \cite{Gaiotto:2014bza}. We then review in Section \ref{RHprob} the Riemann-Hilbert problem of \cite{BridgelandCon} that will be relevant for us, and its solution.  After that, we show in Section \ref{solrh} that the conformal limit satisfies properties related to the Riemann-Hilbert problem, and state the corresponding conjecture.

\subsection{The conformal limit}\label{conflimsec}

We now briefly recall the conformal limit studied in \cite{Gaiotto:2014bza} for the twistor coordinates $\mathcal{X}_{\gamma}(x,\zeta)$. For simplicity, we restrict to the special case of the HK variation of BPS structures associated to the resolved conifold $(M,\Gamma,Z'=2iZ,\Omega)$.\\

To take the conformal limit, one picks $\theta_{f}=0$ (and hence $\theta_{\delta}=0$) and restricts to the complex Lagrangian submanifold $L\subset \mathcal{N}(\theta_{f}=0)$ of Proposition \ref{Lagrangian}, obtained by setting $\theta_{\beta}=\theta_{\beta^{\vee}}=0$. We then introduce a scaling parameter $R>0$ for the central charge $Z\to RZ$, and take a formal limit of $R\to 0$ while maintaining $\lambda:=\zeta/R$ fixed. If $\pi(x)=t$, where $\pi:\mathcal{N}(\theta_f=0)\to M$, we will, in a slight abuse of notation, denote the conformal limit of $\mathcal{X}_{\gamma}(x,\zeta)$ by $\mathcal{X}_{\gamma}(t,\lambda)$.  \\

For the functions $\mathcal{X}_{\beta}(x,\zeta)$ and $\mathcal{X}_{\delta}(x,\zeta)$ the conformal limit always exists (since $\mathcal{X}_{\gamma}=\mathcal{X}_{\gamma}^{\text{sf}}$ for $\gamma=\beta, \delta$)  and is given by
\begin{equation}
    \mathcal{X}_{\beta}(t,\lambda)=\exp\Big(\frac{2\pi i}{\lambda}Z_{\beta}\Big)=\exp\Big(\frac{2\pi i t}{\lambda}\Big), \;\;\;\; \mathcal{X}_{\delta}(t,\lambda)=\exp\Big(\frac{2\pi i}{\lambda}Z_{\delta}\Big)=\exp\Big(\frac{2\pi i }{\lambda}\Big)\, .
\end{equation}

On the other hand, the resulting expression for $\mathcal{X}_{\beta^{\vee}}(t,\lambda)$ is given by \cite[Equation 2.3]{Gaiotto:2014bza}

\begin{equation}\label{conflim}
\begin{split}
    \mathcal{X}_{\beta^{\vee}}(t,\lambda)&=\exp\Big(\frac{2\pi i Z_{\beta^{\vee}}}{\lambda}  -\sum_{n \in \mathbb{Z}}\frac{\lambda}{\pi i} \int_{\mathbb{R}_{-}2iZ_{\beta + n \delta}}\frac{d\lambda'}{(\lambda')^2-(\lambda)^2}\log (1- e^{2\pi i Z_{\beta + n \delta}/\lambda'})\Big)\\
    &=\exp\Big(\frac{2\pi i Z_{\beta^{\vee}}}{\lambda}  -\sum_{n \in \mathbb{Z}} \frac{\lambda}{\pi i}\int_{i\mathbb{R}_{-}(t+n)}\frac{d\lambda'}{(\lambda')^2-(\lambda)^2}\log (1- e^{2\pi i (t+n)/\lambda'})\Big)\,.
\end{split}
\end{equation}
In the previous expression, $\lambda$ should be away from the BPS rays $\{\pm i\mathbb{R}_{-}( t+n)\}_{n \in \mathbb{Z}}\cup i\mathbb{R}_{\pm}$ and we also assume that $\text{Im}(t)\neq 0$ so that the rays $\pm i\mathbb{R}_{-}( t+n)$ are not all collapsed into the rays $i\mathbb{R}_{\pm}$. We furthermore remark that the expression for $\mathcal{X}_{\beta^{\vee}}(t,\lambda)$ is purely formal. Indeed, as we will see below, the infinite sums are not absolutely convergent, but only conditionally convergent. In the next sections, we study this issue, and claim that there is a conditionally convergent expression for $\mathcal{X}_{\beta^{\vee}}(t,\lambda)$ that satisfies properties related (but not entirely equivalent) to the Riemann-Hilbert problem studied in \cite{BridgelandCon}.

\subsection{Review of Riemann-Hilbert problem from Bridgeland}\label{RHprob}

In \cite[Section 3.3]{BridgelandCon}, the following variation of BPS structures $(\widetilde{M},\widetilde{\Gamma},\widetilde{Z},\widetilde{\Omega})$ is associated to the resolved conifold:

\begin{itemize}
    \item $\widetilde{M}$ is a complex $2$-dimensional manifold defined by
    \begin{equation}
        \widetilde{M}:=\{(v,w)\in \mathbb{C}^{2}\;\;\; | \;\;\; w\neq 0 \;\; \text{and} \;\; v+nw \neq 0 \;\; \text{for all} \;\; n \in \mathbb{Z} \}\,.
    \end{equation}
    \item $\widetilde{\Gamma}\to \widetilde{M}$ is given by $\widetilde{\Gamma}=\Lambda \oplus \Lambda^*$, where $\Lambda=\mathbb{Z}\beta \oplus \mathbb{Z}\delta$, and $\Lambda^*=\mathbb{Z}\beta^{\vee}\oplus \mathbb{Z}\delta^{\vee}$ is the dual. The pairing $\langle - , - \rangle$ is given via the duality pairing:
    \begin{equation}
        \langle (\gamma_1,\alpha_1) , (\gamma_2,\alpha_2) \rangle= \alpha_1(\gamma_2)-\alpha_2(\gamma_1)\,.
    \end{equation}
    \item The central charge $\widetilde{Z}$ is defined by
    
    \begin{equation}
        \widetilde{Z}_{n\beta + m\delta}=2\pi i (nv+nw),\;\;\;\; \widetilde{Z}_{n\beta^{\vee}+m\delta^{\vee}}=0 \;\;\; \text{for} \;\;\; {n,m \in \mathbb{Z}} \,.
    \end{equation}
    \item The BPS spectrum is:
    \begin{equation}
\widetilde{\Omega}(\gamma) =
  \begin{cases}
    1       & \quad \text{if } \gamma = \pm \beta + n \delta\quad \text{for }\;\; n\in \mathbb{Z},\\
    -2  & \quad \text{if } \gamma= k \delta \quad \text{for} \;\;k\in \mathbb{Z}\setminus \{ 0\},\\
    0 & \quad \text{otherwise}\,.
  \end{cases}
\end{equation}

\end{itemize}

\begin{rem}\leavevmode
\begin{itemize}
        \item Notice that due to the choice of $\widetilde{Z}_{n\beta^{\vee}+m\delta^{\vee}}=0$, the tuple $(\widetilde{M},\widetilde{\Gamma},\widetilde{Z})$ is not a central charge ASK geometry, and hence $(\widetilde{M},\widetilde{\Gamma},\widetilde{Z},\widetilde{\Omega})$ is not an HK variation of BPS structures. The focus of \cite{BridgelandCon} is, however, on a certain Riemann-Hilbert problem defined by $(\widetilde{M},\widetilde{\Gamma},\widetilde{Z},\widetilde{\Omega})$, in which the values of $\widetilde{Z}_{n\beta^{\vee}+m\delta^{\vee}}$ enter in a non-interesting way. 
    \item If $(M,\Gamma,Z,\Omega)$ is the HK variation of BPS structures associated to the resolved conifold in Section \ref{instconifold}, then $M$ can be embedded in $\widetilde{M}$ via
    \begin{equation}
        t \to (t, 1)\,.
    \end{equation}
\end{itemize}
\end{rem}

\textbf{Riemann-Hilbert problem:} in \cite[Section 3.3]{BridgelandCon} a certain  Riemann-Hilbert problem is associated to  $(\widetilde{M},\widetilde{\Gamma},\widetilde{Z},\widetilde{\Omega})$ for a fixed $(v,w)\in \widetilde{M}$. As previously mentioned in the introduction and Remarks \ref{flavorcontribution}, \ref{flavorcontributionconifold}; the HK geometry we consider only takes into account part of the BPS spectrum of the resolved conifold, with $\Omega(k\delta)=-2$ making no contributions. As a consequence of this, we will only be able to make contact with part of the Riemann-Hilbert problem of \cite{BridgelandCon}. In reviewing the Riemann-Hilbert problem, we will then only focus on the part that will concern us.\\

We assume that $\text{Im}(v/w)>0$, and define the following rays:

\begin{equation}
    \begin{split}
    l_{n}&:=\mathbb{R}_{+}\cdot \widetilde{Z}_{\beta +n \delta}=\mathbb{R}_{+}\cdot(2\pi i(v +nw))\subset \mathbb{C}^{\times}\,,\\
    l_{\infty}&:=\mathbb{R}_{+}\cdot \widetilde{Z}_{\delta}=\mathbb{R}_{+}\cdot 2\pi iw\subset \mathbb{C}^{\times}\,.
    \end{split}
\end{equation}

\begin{dfn} The rays $\pm l_n$ for $n \in \mathbb{Z}$ and $\pm l_{\infty}$, will be called BPS rays.

\end{dfn}

For each ray $l$ different from the BPS rays we then want to find a holomorphic function $\Phi_{l,\beta^{\vee}}(v,w,-):\mathbb{H}_l\to \mathbb{C}^{\times}$ where $\mathbb{H}_{l}$ is the half plane centered at $l$. The functions  $\Phi_{l,\beta^{\vee}}(v,w,-)$ should satisfy the following:

\begin{itemize}
    \item (RH1) If $l$ lies between $l_{n-1}$ and $l_{n}$ and $l'$ lies between $l_n$ and $l_{n+1}$  then  on $\mathbb{H}_{l}\cap \mathbb{H}_{l'} $
    
    \begin{equation}\label{jump1}
        \Phi_{l,\beta^{\vee}}(v,w,\lambda)
        =\Phi_{l',\beta^{\vee}}(v,w,\lambda)(1-e^{-2\pi i(v+nw)/\lambda})\,.
    \end{equation}
    On the other hand, if $l$ lies between $-l_{n-1}$ and $-l_{n}$ and $l'$ lies between $-l_n$ and $-l_{n+1}$  then  on $\mathbb{H}_{l}\cap \mathbb{H}_{l'} $
        \begin{equation}\label{jump2}
        \Phi_{l,\beta^{\vee}}(v,w,\lambda)
        =\Phi_{l',\beta^{\vee}}(v,w,\lambda)(1-e^{2\pi i(v+nw)/\lambda})^{-1}\,.
    \end{equation}
    Furthermore, if $l$ lies between $l_0$ and $l_1$, while $l'$ lies between $-l_0$ and $-l_{-1}$ then we have on $\mathbb{H}_{l}\cap \mathbb{H}_{l'} $
    \begin{equation}\label{jump3}
        \Phi_{l,\beta^{\vee}}(v,w,\lambda)
        =\Phi_{l',\beta^{\vee}}(v,w,\lambda)\prod_{n>0 }(1-e^{-2\pi i( v+nw)/\lambda})\prod_{n>0} (1-e^{2\pi i( v-nw)/\lambda})^{-1}\,. \end{equation}
    Finally,  if $l$ lies between $l_0$ and $l_{-1}$, while $l'$ lies between $-l_0$ and $-l_1$ then we have on $\mathbb{H}_{l}\cap \mathbb{H}_{l'} $
    
    \begin{equation}\label{jump4}
        \Phi_{l',\beta^{\vee}}(v,w,\lambda)
        =\Phi_{l,\beta^{\vee}}(v,w,\lambda)\prod_{n> 0}(1-e^{-2\pi i( v-nw)/\lambda})\prod_{n>0} (1-e^{2\pi i( v+nw)/\lambda})^{-1} \,. \end{equation}
    
    \item (RH2) Given any ray $l$ away from the BPS rays,  $\Phi_{l,\beta^{\vee}}(v,w,-)$ satisfies 
    
    \begin{equation}
        \Phi_{l,\beta^{\vee}}(v,w,\lambda) \to 1  \;\;\;\; \text{as} \;\;\;\; \lambda \to 0, \;\; \lambda \in \mathbb{H}_l  \,.
    \end{equation}

    \item (RH3) With the same notation as the previous point, for any $\mathbb{H}_{l}$ there is $k >0$ such that
    \begin{equation}
    |\lambda|^{-k}<|\Phi_{l,\beta^{\vee}}(v,w,\lambda)|<|\lambda|^{k} \;\;\;\; \text{as} \;\;\;\; \lambda \to \infty, \;\; \lambda \in \mathbb{H}_l  \,.
    \end{equation}

\end{itemize}

A solution to this problem is unique (see for example \cite[Lemma 4.9]{BridgelandDT}). In particular, a solution is given in \cite{BridgelandCon} in terms of the  Faddeev quantum dilogarithm. In order to write it down, we briefly recall how the quantum dilogarithm is defined.

\subsubsection{The Faddeev quantum dilogarithm}\label{FQD}
 The definition involves the multiple sine functions, which in turn are defined using the Barnes multiple Gamma functions $\Gamma_r(z\, |\, \omega_1,...,\omega_r)$ \cite{Barnes}. For a variable $z\in \mathbb{C}$ and parameters $\omega_1,\ldots,\omega_r \in \mathbb{C}^{*}$ the multiple sine functions are defined by:
\begin{equation}
    \sin_r(z\,|\, \omega_1,\dots,\omega_r):= \Gamma_{r}(z\, |\, \omega_1,\dots,\omega_r) \cdot \Gamma_{r}\left(\sum_{i=1}^r \omega_i - z\, |\, \omega_1,\dots,\omega_r\right)^{(-1)^r} \,,
\end{equation}
for further definitions, see e.~g.~ \cite{BridgelandCon,Ruijsenaars1} and references therein. We introduce furthermore  the generalized Bernoulli polynomials, defined by the generating function:
\begin{equation}
    \frac{x^r\, e^{zx}}{ \prod_{i=1}^r (e^{\omega_i x}-1)} = \sum_{n=0}^{\infty} \frac{x^n}{n!} \, B_{r,n}(z\,|\, \omega_1,\,\dots,\omega_r)\,.
\end{equation}
The quantum dilogarithm is then defined by:
\begin{equation}
    H(t \,| \,\omega_1, \omega_2 ) := \exp\left(-\frac{\pi i }{2}\cdot B_{2,2}(t\,|\,\omega_1,\omega_2)\right) \cdot \sin_{2}(t\,|\, \omega_1,\omega_2)\,.
\end{equation}

\begin{prop} \label{qdilogprop} \cite[Prop 4.1]{BridgelandCon} The function $H(t\, | \, \omega_1,\omega_2)$ is a single-valued meromorphic function of variables $t\in\mathbb{C}$ and $\omega_1,\omega_2 \in \mathbb{C}^*$ under the assumption $\omega_1/\omega_2 \notin \mathbb{R}_{<0}$. It is symmetric in the arguments $\omega_1,\omega_2$ and invariant under simultaneous rescaling of all three arguments.

\end{prop}

Using the quantum dilogarithm, we have the following:

\begin{thm}\cite[Theorem 5.2]{BridgelandCon} \label{quantumdilogRH}Take $(v,w)\in M$ with $\text{Im}(v/w)>0$ and consider the solution of the previous Riemann-Hilbert problem $\{\Phi_{l,\beta^{\vee}}(v,w,-)\}_{l}$. If $l$ is a ray in the sector between $l_0$ and $l_{-1}$, then  $\Phi_{l,\beta^{\vee}}(v,w,-)$ satisfies

\begin{equation}
    \Phi_{l,\beta^{\vee}}(v,w,\lambda)=H(v|w,-\lambda)e^{Q_H(v|w,-\lambda)}\,,
\end{equation}
where $H(t|\omega_1,\omega_2)$ is the quantum dilogarithm and 
\begin{equation}
    Q_H(t|\omega_1,\omega_2):=-\frac{\omega_1}{2\pi i\omega_2}\mathrm{Li}_2(e^{2\pi it/\omega_1})-\frac{1}{2}\log(1-e^{2\pi i t/\omega_1})+\frac{\pi}{12}\frac{\omega_2}{\omega_1}\,.
\end{equation}
\end{thm}

\subsection{The conformal limit and relation to the Riemann-Hilbert problem}\label{solrh}

To study the integral terms of (\ref{conflim}), we will use the following expressions involving the Gamma function $\Gamma(z)$ (see also \cite{Gaiotto:2014bza}\footnote{Note that in \cite{Gaiotto:2014bza} there is a sign mistake on the integral term of the corresponding formula $(3.4)$.}): for $\text{Re}(z)>0$ we have 
\begin{equation}\label{loggammaid}
    \log(\Gamma(z))=z(\log(z)-1)+\log\Big(\sqrt{\frac{2\pi}{z}}\Big)-\frac{1}{\pi}\int_0^{\infty}\frac{ds}{s^2+1}\log(1-e^{-2\pi z/s})\,.
\end{equation}

We denote by $\mu(z)$ the function defined for $\text{Re}(z)>0$ and given by

\begin{equation}\label{Binet}
    \begin{split}
    \mu(z)&:=\log(\Gamma(z))-\Big(z-\frac{1}{2}\Big)\log(z) + z -\frac{1}{2}\log(2\pi)\\
    &=-\frac{1}{\pi}\int_0^{\infty}\frac{ds}{s^2+1}\log(1-e^{-2\pi z/s})\,.\\
    \end{split}
\end{equation}

This function is sometimes known as Binet's function.

\begin{lem} \label{intmu} Let $n\in \mathbb{Z}$ and $t\in \mathbb{C}^{\times}$ with $\mathrm{Im}(t)\neq 0$. Furthermore, assume that $\lambda \in \mathbb{C}^{\times}$ is not in $\pm i\mathbb{R}_{-}(t+n)$. We then have
\begin{equation}
    -\frac{\lambda}{\pi i}\int_{i\mathbb{R}_{-}(t+n)}\frac{d\lambda'}{(\lambda')^2-(\lambda)^2}\log (1- e^{2\pi i (t+n)/\lambda'})=\begin{cases}
      \mu\Big(\dfrac{t+n}{\lambda}\Big) \;\;\; \text{if} \;\;\; \mathrm{Re}((t+n)/\lambda)>0\\
      \\
      -\mu\Big(-\dfrac{t+n}{\lambda}\Big)\;\;\; \text{if} \;\;\; \mathrm{Re}((t+n)/\lambda)<0\,.\\
    \end{cases}
\end{equation}

\end{lem}

\begin{proof} This is basically the same idea of \cite{Gaiotto:2014bza} used to explicitly compute the instanton correction terms of the conformal limit. \\

Since $\lambda$ is not in $\pm i\mathbb{R}_{-}(t+n)$, we must have $\text{Re}((t+n)/\lambda)\neq 0$. Assume first that $\lambda$ is such that $\text{Re}((t+n)/\lambda)<0$. We can then deform the contour of 

\begin{equation}
    -\frac{\lambda}{\pi i} \int_{i\mathbb{R}_{-}(t+n)}\frac{d\lambda'}{(\lambda')^2-(\lambda)^2}\log (1- e^{2\pi i(t+n)/\lambda'})
\end{equation}
to $\lambda'=i\lambda s$ for $s\in (0,\infty)$, and obtain

\begin{equation}
    -\frac{\lambda}{\pi i} \int_{i\mathbb{R}_{-}(t+n)}\frac{d\lambda'}{(\lambda')^2-(\lambda)^2}\log (1- e^{2\pi i(t+n)/\lambda'})=\frac{1}{\pi}\int_0^{\infty}\frac{ds}{s^2+1}\log(1-e^{2\pi\frac{t+n}{\lambda s}})\,.
\end{equation}
Taking $z=-(t+n)/\lambda$ we see that $\text{Re}(z)>0$, so 

\begin{equation}
    -\frac{\lambda}{\pi i} \int_{i\mathbb{R}_{-}(t+n)}\frac{d\lambda'}{(\lambda')^2-(\lambda)^2}\log (1- e^{2\pi i(t+n)/\lambda'})=-\mu\Big(-\frac{t+n}{\lambda}\Big)\,.
\end{equation}
 The case $\text{Re}((t+n)/\lambda)>0$ is similar. We deform the contour to $\lambda'=-i\lambda s$
 and obtain

\begin{equation}
    -\frac{\lambda}{\pi i} \int_{i\mathbb{R}_{-}(t+n)}\frac{d\lambda'}{(\lambda')^2-(\lambda)^2}\log (1- e^{2\pi i(t+n)/\lambda'})=-\frac{1}{\pi}\int_0^{\infty}\frac{ds}{s^2+1}\log(1-e^{-2\pi\frac{t+n}{\lambda s}})\,.
\end{equation}
Taking $z=(t+n)/\lambda$ we see that $\text{Re}(z)>0$, and then 
\begin{equation}
    -\frac{\lambda}{\pi i} \int_{i\mathbb{R}_{-}( t+n )}\frac{d\lambda'}{(\lambda')^2-(\lambda)^2}\log (1- e^{2\pi i( t+n )/\lambda'})=\mu\Big(\frac{t+n}{\lambda}\Big)\,.
\end{equation}
\end{proof}
\begin{prop} \label{coordconv} Let $(t,\lambda)\in (\mathbb{C}^{\times})^2$ with $\mathrm{Im}(t)\neq 0$ and $\lambda$ away from the BPS rays $\{\pm i\mathbb{R}_{-}( t+n)\}_{n\in \mathbb{Z}}\cup i\mathbb{R}_{\pm}$. The sum

\begin{equation}\label{formsum}
    -\sum_{n \in \mathbb{Z}} \frac{\lambda}{\pi i}\int_{i\mathbb{R}_{-}(t+n)}\frac{d\lambda'}{(\lambda')^2-(\lambda)^2}\log (1- e^{2\pi i (t+n)/\lambda'})
\end{equation}
does not converge absolutely, but only conditionally. In particular, the expression 

\begin{equation}\label{condconv}
    -\sum_{n> 0} \Big(\frac{\lambda}{\pi i}\int_{i\mathbb{R}_{-}(t+n)}\frac{d\lambda'}{(\lambda')^2-(\lambda)^2}\log (1- e^{2\pi i (t+n)/\lambda'})+\frac{\lambda}{\pi i}\int_{i\mathbb{R}_{-}(t-n)}\frac{d\lambda'}{(\lambda')^2-(\lambda)^2}\log (1- e^{2\pi i (t-n)/\lambda'})\Big)
\end{equation}
converges uniformly in $\lambda$ on compact subsets avoiding the integration contours. \\

Letting $a_n(t,\lambda):=\text{Re}((t+n)/\lambda)$, the tail of the conditionally convergent expression (\ref{condconv}) can be rewritten as follows: for $M>0$ sufficiently big and $n>M$, we either have $a_n>0$ and $a_{-n}<0$, or $a_n<0$ and $a_{-n}>0$

\begin{itemize}
    \item  In the first case, the tail of (\ref{condconv}) is given by 
    \begin{equation}
        \sum_{n>M}\Big(\mu\Big(\frac{t+n}{\lambda}\Big)-\mu\Big(-\frac{t-n}{\lambda}\Big)\Big)\,.
    \end{equation}
    \item In the second case, the tail of (\ref{condconv}) is given by 
    
    \begin{equation}
        \sum_{n>M}\Big(-\mu\Big(-\frac{t+n}{\lambda}\Big)+\mu\Big(\frac{t-n}{\lambda}\Big)\Big)\,.
    \end{equation}
\end{itemize}

\end{prop}

\begin{proof}

It is easy to check that under our hypotheses, for $M>0$ sufficiently big and $n>M$ we either have $a_{n}>0$ and $a_{-n}<0$,  or $a_{n}<0$ and $a_{-n}>0$. We assume the first case, since the second is analogous.  This means that for $n>M$, we can write using Lemma \ref{intmu}:

\begin{equation}
    \begin{split}
    -\sum_{n> M}& \Big(\frac{\lambda}{\pi i}\int_{i\mathbb{R}_{-}(t+n)}\frac{d\lambda'}{(\lambda')^2-(\lambda)^2}\log (1- e^{2\pi i (t+n)/\lambda'})+\frac{\lambda}{\pi i}\int_{i\mathbb{R}_{-}(t-n)}\frac{d\lambda'}{(\lambda')^2-(\lambda)^2}\log (1- e^{2\pi i (t-n)/\lambda'})\Big)\\
    &=\sum_{n>M}\Big(\mu\Big(\frac{t+n}{\lambda}\Big)-\mu\Big(-\frac{t-n}{\lambda}\Big)\Big)\,.
    \end{split}
\end{equation}

We now use the fact that for $\text{Re}(z)>0$, we have the following asymptotic expansion of Binet's function as $z\to \infty$:

\begin{equation}
    \mu(z)=\sum_{m=1}^n\frac{B_{2m}}{(2m-1)2mz^{2m-1}}+ \mathcal{O}(|z|^{1-2n})\,,
\end{equation}
where $B_{n}$ are the Bernoulli numbers.\\

In particular, for $n>M$

\begin{equation}
    \mu\Big(\frac{t+n}{\lambda}\Big)-\mu\Big(-\frac{t-n}{\lambda}\Big)=\frac{\lambda tB_2}{t^2-n^2} + \frac{B_4 \lambda^3}{12 (t+n)^3}-\frac{B_4 \lambda^3}{12 (n-t)^3}+\mathcal{O}\Big(\frac{\lambda^3}{n^3}\Big)\,,
\end{equation}
where we used that

\begin{equation}\label{diff}
    \frac{\lambda B_2}{2(t+n)}-\frac{\lambda B_2}{2(-t+n)}=\frac{\lambda tB_2}{t^2-n^2}\,.
\end{equation}
We then conclude that 
\begin{equation}
    \sum_{n>M}\Big(\mu\Big(\frac{t+n}{\lambda}\Big)-\mu\Big(-\frac{t-n}{\lambda}\Big)\Big)
\end{equation}
converges uniformly in $\lambda$ for compact subsets avoiding the integration contours.\\

On the other hand, since
\begin{equation}
    |\mu(z)|=\mathcal{O}(|z|^{-1})\,,
\end{equation}
it is easy to check that (\ref{formsum}) cannot converge absolutely. 
\end{proof}

\begin{rem}
From now on, every time we write $\mathcal{X}_{\beta^{\vee}}^{\text{inst}}(t,\lambda)$ we will mean the convergent expression (\ref{condconv}) from Proposition \ref{coordconv}. Namely, 

\begin{equation}
    \begin{split}
        \log(\mathcal{X}_{\beta^{\vee}}^{\text{inst}}(t,\lambda))&=-\frac{\lambda}{\pi i}\int_{i\mathbb{R}_{-}t}\frac{d\lambda'}{(\lambda')^2-(\lambda)^2}\log (1- e^{2\pi i t/\lambda'})\\
            &\;\;\;\;-\sum_{n> 0} \Big[\frac{\lambda}{\pi i}\int_{i\mathbb{R}_{-}(t+n)}\frac{d\lambda'}{(\lambda')^2-(\lambda)^2}\log (1- e^{2\pi i (t+n)/\lambda'})\\
            &\;\;\;\;\;\;\;\;\;\;\;+\frac{\lambda}{\pi i}\int_{i\mathbb{R}_{-}(t-n)}\frac{d\lambda'}{(\lambda')^2-(\lambda)^2}\log (1- e^{2\pi i (t-n)/\lambda'})\Big]\\
    \end{split}
\end{equation}
\end{rem}

\begin{dfn}\label{halfplanecont}Given $\mathcal{X}_{\beta^{\vee}}(t,\lambda)$  and a ray $l$ not equal to the BPS rays, we can define a holomorphic function $\mathcal{X}_{l,\beta^{\vee}}(t,-):\mathbb{H}_{l}\to \mathbb{C}^{\times}$ as follows:

\begin{itemize}
    \item Let $\Sigma \subset \mathbb{H}_l$ be any closed subsector containing $l$.  For the BPS rays contained in $\mathbb{H}_l$, we can deform them to $\mathbb{H}_l-\Sigma$ without crossing $l$ and maintaining the relative ordering. For the BPS rays contained in $-\mathbb{H}_l$ we similarly deform them to $-(\mathbb{H}_l-\Sigma)$ without crossing $-l$ and maintaining the relative ordering. 
    \item This defines an analytic continuation $\mathcal{X}_{\Sigma,\beta^{\vee}}(t,-):\Sigma \to \mathbb{C}^{\times}$ of $\mathcal{X}_{\beta^{\vee}}(t,-)$, coinciding with $\mathcal{X}_{\beta^{\vee}}(t,-)$ on the sector made of two consecutive BPS rays containing $l$.
\end{itemize}
Since we can do this for any closed subsector $\Sigma \subset \mathbb{H}_l$, we can define an analytic continuation $\mathcal{X}_{l,\beta^{\vee}}(t,-):\mathbb{H}_{l}\to \mathbb{C}^{\times}$ of $\mathcal{X}_{\beta^{\vee}}(t,-)$. 
\end{dfn}

In order to try to relate $\mathcal{X}_{\beta^{\vee}}(t,\lambda)$ to the solution of the Riemann-Hilbert problem from \cite{BridgelandCon}, we would like to show that the functions $\mathcal{X}_{l,\beta^{\vee}}(t,\lambda)$  satisfy properties related to (RH1), (RH2) and (RH3) from Section \ref{RHprob}.\\

Assuming $\text{Im}(t)>0$, we consider the rays:

\begin{equation}
    \begin{split}
    l_{n}&:=\mathbb{R}_{+}\cdot 2iZ_{\beta +n \delta}=i\mathbb{R}_{+}\cdot(t+n)\subset \mathbb{C}^{\times}\,,\\
    l_{\infty}&:=\mathbb{R}_{+}\cdot 2iZ_{\delta}=i\mathbb{R}_{+}\subset \mathbb{C}^{\times}\,.
    \end{split}
\end{equation}
Then the BPS rays are precisely $\pm l_{n}$ for $n \in \mathbb{Z}$, and $\pm l_{\infty}$.

\begin{prop}\label{jumpsprop}Fix $t$ with $\text{Im}(t)>0$.  If $l$ is a ray between $l_{n}$ and $l_{n-1}$ and $l'$ a ray between $l_{n+1}$ and $l_{n}$ then we have
    
    \begin{equation}\label{jumps}
        \begin{split}
        \mathcal{X}_{-l,\beta^{\vee}}(t,\lambda)
        &=\mathcal{X}_{-l',\beta^{\vee}}( t,\lambda)(1-e^{2\pi i( t+n)/\lambda})^{-1}, \;\;\;\; \text{for} \;\;\;\; \lambda\in\mathbb{H}_{-l}\cap \mathbb{H}_{-l'}\\
        \mathcal{X}_{l,\beta^{\vee}}(t,\lambda)
        &=\mathcal{X}_{l',\beta^{\vee}}( t,\lambda)(1-e^{-2\pi i( t+n)/\lambda}), \;\;\;\; \text{for} \;\;\;\;  \lambda\in  \mathbb{H}_l\cap \mathbb{H}_{l'}\,.\\
        \end{split}
    \end{equation}
    If $l$ lies between $l_0$ and $l_1$, while $l'$ lies between $-l_0$ and $-l_{-1}$ then on $\mathbb{H}_{l}\cap \mathbb{H}_{l'} $ we have 
    \begin{equation}\label{infprod1}
        \mathcal{X}_{l,\beta^{\vee}}(t,\lambda)
        =\mathcal{X}_{l',\beta^{\vee}}(t,\lambda)\prod_{n>0 }(1-e^{-2\pi i( t+n)/\lambda})\prod_{n>0} (1-e^{2\pi i( t-n)/\lambda})^{-1}\,.  \end{equation}
    Finally,  if $l$ lies between $l_0$ and $l_{-1}$, while $l'$ lies between $-l_0$ and $-l_{1}$ then on $\mathbb{H}_{l}\cap \mathbb{H}_{l'} $ we have:
    
    \begin{equation}\label{infprod2}
        \mathcal{X}_{l',\beta^{\vee}}(t,\lambda)
        =\mathcal{X}_{l,\beta^{\vee}}(t,\lambda)\prod_{n>0} (1-e^{-2\pi i (t-n)/\lambda})\prod_{n> 0}(1-e^{2\pi i( t+n)/\lambda})^{-1} \,. \end{equation}

\end{prop}

\begin{rem} Similar formulas follow for the case $\text{Im}(t)<0$, but we will only need the case $\text{Im}(t)>0$.

\end{rem}

\begin{proof}  Let us first show the first statement. From the definition \ref{halfplanecont} of the analytic continuations $\mathcal{X}_{l,\beta^{\vee}}$, we can compare $\mathcal{X}_{-l,\beta^{\vee}}/\mathcal{X}_{-l',\beta^{\vee}}$ for $\lambda\in\mathbb{H}_{-l}\cap \mathbb{H}_{-l'}$ by computing the contour

\begin{equation}
    \log(\mathcal{X}_{-l,\beta^{\vee}}(t,\lambda)/\mathcal{X}_{-l',\beta^{\vee}}(t,\lambda))=-\frac{\lambda}{\pi i}\oint \frac{d\lambda'}{(\lambda')^2-(\lambda)^2}\log(1-e^{2\pi i(t+n)/\lambda'}) \,,
\end{equation}
where we integrate along a small (counter-clockwise) contour around $\lambda$.\\

On the other hand,

\begin{equation}
-\frac{\lambda}{\pi i}\oint \frac{d\lambda'}{(\lambda')^2-(\lambda)^2}\log(1-e^{2\pi i(t+n)/\lambda'})=-\log (1-e^{2\pi i(t+n)/\lambda})\,,
\end{equation}
so 
    \begin{equation}
        \mathcal{X}_{-l,\beta^{\vee}}(t,\lambda)
        =\mathcal{X}_{-l',\beta^{\vee}}( t,\lambda)(1-e^{2\pi i( t+n)/\lambda})^{-1}\,.
    \end{equation}
    
Similarly, since 

\begin{equation}\label{changeint}
    -\frac{\lambda}{\pi i} \int_{i\mathbb{R}_{-}(t+n)}\frac{d\lambda'}{(\lambda')^2-(\lambda)^2}\log (1- e^{2\pi i(t+n)/\lambda'})=\frac{\lambda}{\pi i}\int_{i\mathbb{R}_{-}(-t-n)}\frac{d\lambda'}{(\lambda')^2-(\lambda)^2}\log(1-e^{2\pi i(-t-n)/\lambda'})\,,
\end{equation}
where $i\mathbb{R}_{-}(-t-n)=l_n$, then for $\lambda\in  \mathbb{H}_l\cap \mathbb{H}_{l'}$
\begin{equation}
    \log(\mathcal{X}_{l,\beta^{\vee}}(t,\lambda)/\mathcal{X}_{l',\beta^{\vee}}(t,\lambda))=\frac{\lambda}{\pi i}\oint \frac{d\lambda'}{(\lambda')^2-(\lambda)^2}\log(1-e^{2\pi i(-t-n)/\lambda'})=\log(1-e^{-2\pi i(t+n)/\lambda})\,, 
\end{equation}
which implies that 
\begin{equation}
        \mathcal{X}_{l,\beta^{\vee}}(t,\lambda)
        =\mathcal{X}_{l',\beta^{\vee}}( t,\lambda)(1-e^{-2\pi i(t+n)/\lambda})\,.
\end{equation}

Now assume that the ray $l$ lies between $l_0$ and $l_1$, while $l'$ lies between $-l_0$ and $-l_{-1}$. Then to relate $\mathcal{X}_{l,\beta^{\vee}}(t,\lambda)$ and $\mathcal{X}_{l',\beta^{\vee}}(t,\lambda)$ one needs to compute and infinite amount of residues corresponding to (after using \eqref{changeint}) the contour integrals along $l_n$ and $-l_{-n}$ for $n>0$, giving for $\lambda\in \mathbb{H}_l\cap \mathbb{H}_{l'}$

\begin{equation}
    \begin{split}
    \log(\mathcal{X}_{l,\beta^{\vee}}(t,\lambda)/\mathcal{X}_{l',\beta^{\vee}}(t,\lambda))=\sum_{n>0}&\bigg(\frac{\lambda}{\pi i}\oint \frac{d\lambda'}{(\lambda')^2-(\lambda)^2}\log(1-e^{-2\pi i(t+n)/\lambda'})\\
    &-\frac{\lambda}{\pi i}\oint \frac{d\lambda'}{(\lambda')^2-(\lambda)^2}\log(1-e^{2\pi i(t-n)/\lambda'})\bigg)\\
    =\sum_{n>0}&\bigg(\log(1-e^{-2\pi i(t+n)/\lambda})-\log (1-e^{2\pi i(t-n)/\lambda})\bigg)
    \end{split}
\end{equation}
We remark that because $\lambda \in \mathbb{H}_l\cap \mathbb{H}_{l'}$, we have $\text{Re}(-2\pi i(t+n)/\lambda)<0$ and $\text{Re}(2\pi i(t-n)/\lambda)<0$. Indeed for $\lambda\in \mathbb{H}_{l_n}$ we have $\text{Re}(-2\pi i(t+n)/\lambda)<0$ and for $\lambda \in \mathbb{H}_{-l_{-n}}$ we have $\text{Re}(2\pi i(t-n)/\lambda)<0$, while $\mathbb{H}_l\cap \mathbb{H}_{l'}$ intersects $\mathbb{H}_{l_n}$ and $\mathbb{H}_{-l_{-n}}$ for all $n>0$. Hence, each term in the above infinite sum exponentially decays for $\lambda\in \mathbb{H}_l\cap \mathbb{H}_{l'}$ as $n\to \infty$, and hence the sum converges. In particular, we find that

 \begin{equation}
        \mathcal{X}_{l,\beta^{\vee}}(t,\lambda)
        =\mathcal{X}_{l',\beta^{\vee}}(t,\lambda)\prod_{n>0 }(1-e^{-2\pi i( t+n)/\lambda})\prod_{n>0} (1-e^{2\pi i( t-n)/\lambda})^{-1}\,.  \end{equation}
A similar argument follows for (\ref{infprod2}).

\end{proof}

\begin{rem} The above proposition show that the analytic continuations $\mathcal{X}_{l,\beta^{\vee}})(t,-)$ satisfy (RH1) from Section \ref{RHprob}. 
\end{rem}

We now prove the following proposition regarding the asymptotics of the Riemann-Hilbert problem as $\lambda \to 0$, and related to (RH2) of Section \ref{RHprob}.  Notice that the result below is only shown along sectors determined by BPS rays.

\begin{prop}\label{asymptotic1} Take $t$ with $\text{Im}(t)\neq 0$ and consider a ray $l$ not equal to the any of the BPS rays $\{\pm i\mathbb{R}_{-}( t +n)\}_{n \in \mathbb{Z}}\cup i\mathbb{R}_{\pm}$. Let $S_l$ be any sector containing $l$ and no BPS rays. Then 

\begin{equation}
    \mathcal{X}_{\beta^{\vee}}(t,\lambda)e^{-2\pi i Z_{\beta^{\vee}}/\lambda}\to 1 \;\;\;\; \text{as} \;\;\;\; \lambda \to 0, \;\; \lambda \in S_{l}\,.   
\end{equation}

\end{prop}
\begin{proof}
As before, we denote $a_n(t,\lambda)=\text{Re}((t+n)/\lambda)$. Notice that for any $\lambda \in S_l$ we must have $a_n(t,\lambda)\neq 0$ for all $n \in \mathbb{Z}$ (otherwise $S_l$ would contain a BPS ray). On the other hand, since $S_l$ is connected and $\lambda \to a_n(t,\lambda)$ continuous, we have either $a_n(t,\lambda)>0$ or $a_n(t,\lambda)<0$ for all $\lambda \in S_l$. Furthermore, by picking $M>0$ sufficiently big, we have that for $n>M$ either $a_{n}>0$ and $a_{-n}<0$,  or $a_{n}<0$ and $a_{-n}>0$. We assume the first case, since the second is analogous. We can then write for $\lambda \in S_l$

\begin{equation}
    \begin{split}
    \log(\mathcal{X}_{\beta^{\vee}}(t,\lambda)e^{-2\pi i Z_{\beta^{\vee}}/\lambda})&=
    \sum_{a_n>0, |n|\leq M}\mu\Big(\frac{t+n}{\lambda}\Big)-\sum_{a_n<0,|n|\leq M}\mu\Big(-\frac{t+n}{\lambda}\Big)\\
    &+ \sum_{n>M}\Big(\mu\Big(\frac{t+n}{\lambda}\Big)-\mu\Big(-\frac{t-n}{\lambda}\Big)\Big)\,.\\
    \end{split}
\end{equation}

To deal with the terms in the finite sums, we use that
\begin{equation}
    \mu(z)= \frac{B_2}{2z} + \mathcal{O}(|z|^{-1}) \;\;\;\; \text{as} \;\;\;\; z\to \infty \;\;\;\; \text{with} \;\;\;\; \text{Re}(z)>0\,,
\end{equation}
where $z$ is either $z=(t+n)/\lambda$ for the terms with $a_n>0$ or $z=-(t+n)/\lambda$ for the terms with $a_n<0$. This shows that 
\begin{equation}
    \lim_{\lambda \to 0, \lambda \in S_l} \sum_{a_n>0, |n|\leq M}\mu\Big(\frac{t+n}{\lambda}\Big)-\sum_{a_n<0,|n|\leq M}\mu\Big(-\frac{t+n}{\lambda}\Big)=0\,.
\end{equation}

We now deal with the infinite sums. We use again the fact that for $\text{Re}(z)>0$, we can write

\begin{equation}
    \mu(z)=\sum_{m=1}^n\frac{B_{2m}}{(2m-1)2mz^{2m-1}}+ \mathcal{O}(|z|^{1-2n})\,.
\end{equation}

From this and equation (\ref{diff}), it follows that if $U_0\subset \mathbb{C}$ denotes a small neighborhood of $0$, we have that for all $n>M$ and $\lambda \in S_{l}\cap U_0$, the following uniform estimate in $\lambda$ holds:

\begin{equation}
        \Big|\mu\Big(\frac{t+n}{\lambda}\Big)-\mu\Big(-\frac{t-n}{\lambda}\Big)\Big|=\mathcal{O}\Big(\Big| \frac{\lambda}{n^2}\Big|\Big)=\mathcal{O}(n^{-2})\,.
\end{equation}

It follows that we can apply the dominated convergence theorem to interchange limits and infinite sums, and conclude that

\begin{equation}
    \lim_{\lambda \to 0, \lambda \in S_l}\sum_{n>M}\Big(\mu\Big(\frac{t+n}{\lambda}\Big)-\mu\Big(-\frac{t-n}{\lambda}\Big)\Big)=\sum_{n>M}\lim_{\lambda \to 0, \lambda \in S_l}\Big(\mu\Big(\frac{t+n}{\lambda}\Big)-\mu\Big(-\frac{t-n}{\lambda}\Big)\Big)=0\,.
\end{equation}

Hence, putting all together one finds

\begin{equation}
    \lim_{\lambda \to 0, \lambda \in S_l}\log(\mathcal{X}_{\beta^{\vee}}(t,\lambda)e^{-2\pi iZ_{\beta^{\vee}}/\lambda})=0\,.
\end{equation}

\end{proof}

Finally, we deal with a property related to (RH3) of Section \ref{RHprob}.  As with Proposition \ref{asymptotic1}, the Proposition below is only shown for sectors determined by BPS rays.

\begin{prop}\label{asymptotic2} With the same hypotheses as in Proposition \ref{asymptotic1}, there is $k>0$ such that

\begin{equation}\label{estimate3}
    |\lambda|^{-k}<|\mathcal{X}_{\beta^{\vee}}(t,\lambda)|<|\lambda|^k \;\;\;\;\; \text{as} \;\;\;\; \lambda \to \infty, \;\;\; \lambda \in S_l \,.
\end{equation}
\end{prop}
\begin{proof} 

We assume as before that we pick $M>0$ such that for $|n|>M$, the $a_n=\text{Re}((t+n)/\lambda)$ have a definite sign. We assume for definiteness that $a_n>0$ for $n>M$ (and hence $a_{-n}<0$), with the other case being analogous. We can then write as before

\begin{equation}
    \begin{split}
    \log(\mathcal{X}_{\beta^{\vee}}(t,\lambda))&=2\pi iZ_{\beta^{\vee}}(t)/\lambda +\sum_{a_n>0, |n|\leq M}\mu\Big(\frac{t+n}{\lambda}\Big)-\sum_{a_n<0,|n|\leq M}\mu\Big(-\frac{t+n}{\lambda}\Big)\\
    &+ \sum_{n>M}\Big(\mu\Big(\frac{t+n}{\lambda}\Big)-\mu\Big(-\frac{t-n}{\lambda}\Big)\Big)\,.\\
    \end{split}
\end{equation}
On one hand, expanding $\mu(z)$ as $z \to 0$ we find that as $\lambda \to \infty$ with $\lambda \in S_l$ we get the following depending on whether $a_n>0$ or $a_n<0$, respectively:

\begin{equation}
    \begin{split}
    \mu\Big(\frac{t+n}{\lambda}\Big)&= -\frac{1}{2}\log\Big(2\pi\frac{t+n}{\lambda}\Big) + \Big(1-\gamma -\log\Big(\frac{t+n}{\lambda }\Big)\Big)\frac{t+n}{\lambda }+ \mathcal{O}(\Big(\frac{t+n}{\lambda}\Big)^2)\,,\\
    -\mu\Big(-\frac{t+n}{\lambda}\Big)&= \frac{1}{2}\log\Big(-2\pi\frac{t+n}{\lambda}\Big)-(1-\gamma-\log\Big(-\frac{t+n}{\lambda }\Big)\Big)\Big(-\frac{t+n}{\lambda }\Big)+ \mathcal{O}(\Big(-\frac{t+n}{\lambda}\Big)^2)\,,
    \end{split}
\end{equation}
where $\gamma$ is the Euler-Mascheroni constant. This implies that as $\lambda \to \infty$ with $\lambda \in S_l$
\begin{equation}\label{estimate2}
    \begin{split}
    \exp\Big( \sum_{a_n>0, |n|\leq M}\mu\Big(\frac{t+n}{\lambda}\Big)-\sum_{a_n<0,|n|\leq M}\mu\Big(-\frac{t+n}{\lambda}\Big)\Big)\sim C(t,M)\lambda^{m/2}\,,
    \end{split}
\end{equation}
where $m:=\#\{-M<n<M: a_n>0\}-\#\{-M<n<M: a_n<0\}$ and $C(t,M)$ is a factor only depending on $t$ and $M$.\\

We now show that 

\begin{equation}
    \sum_{n>M}\Big(\mu\Big(\frac{t+n}{\lambda}\Big)-\mu\Big(-\frac{t-n}{\lambda}\Big)\Big)= \mathcal{O}(\log|\lambda|) \;\;\;\; \text{as} \;\;\;\; \lambda \to \infty, \;\;\; \lambda \in S_l\,.
\end{equation}

 To do this, we will need the following Binet integral representation 
\begin{equation}
    \mu(z)=\frac{1}{2}\int_{0}^{\infty}\frac{ds}{s}\Big(\frac{1+e^{-s}}{1-e^{-s}} -\frac{2}{s}\Big)e^{-zs} \;\;\;\;\; \text{for} \;\;\;\; \text{Re}(z)>0\,.
\end{equation}
We then have

\begin{equation}\label{FT}
    \begin{split}
        \sum_{n>M}\mu\Big(\frac{t+n}{\lambda}\Big)-\mu\Big(-\frac{t-n}{\lambda}\Big)&=\sum_{n>M}\frac{1}{2}\int_{0}^{\infty}\frac{ds}{s}\Big(\frac{1+e^{-s}}{1-e^{-s}} -\frac{2}{s}\Big)(e^{-\frac{t+n}{\lambda}s}-e^{\frac{t-n}{\lambda}s})\\
        &=\sum_{n>M}\int_{0}^{\infty}\frac{ds}{s}\Big(\frac{1+e^{-s}}{1-e^{-s}} -\frac{2}{s}\Big)\sinh\Big(-\frac{ts}{\lambda}\Big)e^{-\frac{ns}{\lambda}}\,.
    \end{split}
\end{equation}

Letting $b(\lambda):=\text{Re}(1/\lambda)$, we will divide the above problem into the cases $b(\lambda)>0$ and $b(\lambda)<0$ (the case $b(\lambda)=0$ does not occur because $\lambda$ does not lie in the rays $i\mathbb{R}_{\pm}$). Notice that if we show (\ref{estimate3}) for $b(\lambda)>0$, we automatically have (\ref{estimate3}) for the case $b(\lambda)<0$, due to the easy to check identity $\mathcal{X}_{\beta^{\vee}}(t,\lambda)=\mathcal{X}_{\beta^{\vee}}^{-1}(t,-\lambda)$. We then restrict to the case $b(\lambda)>0$.\\

We would like to apply the Fubini-Tonelli theorem to interchange sums and integrals in (\ref{FT}). Since $b(\lambda)>0$, then it is easy to check that (provided $M$ is big enough)

\begin{equation}
    \begin{split}
    \int_0^{\infty}&\sum_{n>M}\frac{ds}{s}\Big|\frac{1+e^{-s}}{1-e^{-s}}-\frac{2}{s}\Big|\Big|\sinh\Big(-\frac{ts}{\lambda}\Big)\Big|e^{-nb(\lambda)s}\\
    &= \int_0^{\infty}\frac{ds}{s}\Big|\frac{1+e^{-s}}{1-e^{-s}}-\frac{2}{s}\Big|\Big|\sinh\Big(-\frac{ts}{\lambda}\Big)\Big|\frac{e^{-(M+1)b(\lambda)s}}{1-e^{-b(\lambda)s}}<\infty \,.
    \end{split}
\end{equation}
By the Fubini-Tonelli theorem, we can then interchange sums and integrals in (\ref{FT}) and obtain
\begin{equation}
\begin{split}
\sum_{n>M}\mu\Big(\frac{t+n}{\lambda}\Big)-\mu\Big(-\frac{t-n}{\lambda}\Big)
    &=\int_{0}^{\infty}\frac{ds}{s}\Big(\frac{1+e^{-s}}{1-e^{-s}} -\frac{2}{s}\Big)\sinh\Big(-\frac{ts}{\lambda}\Big)\sum_{n>M}e^{-\frac{ns}{\lambda}}\\
    &=\int_{0}^{\infty}\frac{ds}{s}\Big(\frac{1+e^{-s}}{1-e^{-s}} -\frac{2}{s}\Big)\sinh\Big(-\frac{ts}{\lambda}\Big)\frac{e^{-\frac{(M+1)s}{\lambda}}}{1-e^{-s/\lambda}}\,.
\end{split}
\end{equation}

We now bound the integral from $1$ to $\infty$ and from $0$ to $1$ separately.  Recalling the notation $a_n(t,\lambda)=\text{Re}((t+n)/\lambda)$, we have

\begin{equation}\label{estimate4}
    \begin{split}
    \Big|\int_{1}^{\infty}&\frac{ds}{s}\Big(\frac{1+e^{-s}}{1-e^{-s}} -\frac{2}{s}\Big)\sinh\Big(-\frac{ts}{\lambda}\Big)\frac{e^{-\frac{(M+1)s}{\lambda}}}{1-e^{-s/\lambda}}\Big| \\
    &\leq  C\int_1^{\infty}\frac{ds}{s}|\sinh\Big(\frac{ts}{\lambda}\Big)e^{\frac{-(M+1)s}{\lambda}}| \leq \frac{C}{2}\int_1^{\infty} \frac{ds}{s}e^{-a_{M+1}(t,\lambda)s}+\frac{C}{2}\int_1^{\infty} \frac{ds}{s}e^{-a_{M+1}(-t,\lambda)s}\,.\\
    \end{split}
\end{equation}
Now notice that $a_{M+1}(t,\lambda)>0$ by our initial assumption, while we can ensure that $a_{M+1}(-t,\lambda)>0$ by possibly inscreasing the value of $M$. Hence, the last two integrals in (\ref{estimate4}) can be written in terms of the exponential integral function $E_1(z)$. Using that $E_1(x)<e^{-x}\log(1+1/x)<\log(1+1/x)$ for $x>0$ we then find
\begin{equation}
    \begin{split}
        \Big|\int_{1}^{\infty}&\frac{ds}{s}\Big(\frac{1+e^{-s}}{1-e^{-s}} -\frac{2}{s}\Big)\sinh\Big(-\frac{ts}{\lambda}\Big)\frac{e^{-\frac{(M+1)s}{\lambda}}}{1-e^{-s/\lambda}}\Big|\\
        & \leq \frac{C}{2}\Big(E_1(a_{M+1}(t,\lambda)) + E_1(a_{M+1}(-t,\lambda))\Big )\\
    &\leq \frac{C}{2}\Big(\log\Big(1+\frac{1}{a_{M+1}(t,\lambda)}\Big) + \log\Big(1+\frac{1}{a_{M+1}(-t,\lambda)}\Big)\Big)=\mathcal{O}(\log(|\lambda|))\,.
    \end{split}
\end{equation}
Finally, notice that as $\lambda \to \infty$ with $\lambda \in S_l$ we have 

\begin{equation}
    \Big|\int_{0}^{1}\frac{ds}{s}\Big(\frac{1+e^{-s}}{1-e^{-s}} -\frac{2}{s}\Big)\sinh\Big(-\frac{ts}{\lambda}\Big)\frac{e^{-\frac{(M+1)s}{\lambda}}}{1-e^{-s/\lambda}}\Big|=\mathcal{O}(\lambda^0) \,.
\end{equation}

We conclude that in the case of $b(\lambda)>0$, we have 

\begin{equation}\label{estimate1}
    \sum_{n>M}\Big(\mu\Big(\frac{t+n}{\lambda}\Big)-\mu\Big(-\frac{t-n}{\lambda}\Big)\Big)= \mathcal{O}(\log|\lambda|) \;\;\;\; \text{as} \;\;\;\; \lambda \to \infty, \;\;\; \lambda \in S_l \,,
\end{equation}
so taking $k_1>\text{max}\{m/2+1,0\}$, we find  $k_1>0$ such that
\begin{equation}
    |\mathcal{X}_{\beta^{\vee}}(t,\lambda)|<|\lambda|^{k_1} \;\;\;\;\; \text{as} \;\;\;\; \lambda \to \infty, \;\;\; \lambda \in S_l \,.
\end{equation}

Now we prove the reverse inequality. From (\ref{estimate1}) we find that for some constant $C>0$ 

\begin{equation}
    -C\log|\lambda|<\text{Re}\Big(\sum_{n>M}\Big(\mu\Big(\frac{t+n}{\lambda}\Big)-\mu\Big(-\frac{t-n}{\lambda}\Big)\Big)\Big)\,,
\end{equation}
which together with (\ref{estimate2}) allows us to conclude that for some $k_2>0$
\begin{equation}
    |\lambda|^{-k_2}<|\mathcal{X}_{\beta^{\vee}}(t,\lambda)| \;\;\;\;\; \text{as} \;\;\;\; \lambda \to \infty, \;\;\; \lambda \in S_l\,.
\end{equation}
Taking $k=\text{max}\{k_1,k_2\}$ one then obtains $k>0$ such that 

\begin{equation}
    |\lambda|^{-k}<|\mathcal{X}_{\beta^{\vee}}(t,\lambda)|<|\lambda|^k \;\;\;\;\; \text{as} \;\;\;\; \lambda \to \infty, \;\;\; \lambda \in S_l \,.
\end{equation}

As previously mentioned, the same inequalities follow for the case $b(\lambda)<0$ by using the identity $\mathcal{X}_{\beta^{\vee}}(t,\lambda)=\mathcal{X}_{\beta^{\vee}}^{-1}(t,-\lambda)$.

\end{proof}

 We remark that the asymptotic properties shown above for $\mathcal{X}_{\beta^{\vee}}(t,\lambda)$ are in principle not enough to conclude that the analytic continuations $\mathcal{X}_{l,\beta^{\vee}}(t,-):\mathbb{H}_l\to \mathbb{C}^{\times}$ match Bridgeland's solution $\Phi_{l,\beta^{\vee}}(t,1,\lambda)$ of the Riemann-Hilbert problem. The main issue being that the proof of uniqueness of solutions to the Riemann-Hilbert problem really does use that the asymptotic properties (RH2) and (RH3) hold on the half planes $\mathbb{H}_l$ (or at least on sectors centered at $l$ of fixed opening independent of $l$). On the other hand, our proofs of Propositions \ref{asymptotic1} and \ref{asymptotic2} don't extend easily to a proof for the analytic continuations $\mathcal{X}_{l,\beta^{\vee}}(t,\lambda)$, since the estimates of the Binet functions $\mu(z)$ that we use to represent $\mathcal{X}_{\beta^{\vee}}$ on a given sector really use that we stay on a sector determined by BPS rays (to guarantee that $\text{Re}(z)>0$). Nevertheless, the fact that (RH1) and weaker versions of (RH2) and (RH3) holds for $\mathcal{X}_{\beta^{\vee}}(t,-)$ allows us to conjecture the following possible relation to the solution of Bridgeland's Riemann-Hilbert problem:

\begin{dfn}\label{Bsol} Assume that $(v,w)\in \widetilde{M}$ satisfies $\text{Im}(v/w)>0$, and let $\{\Phi_{l,\beta^{\vee}}\}_l$ be the solution of \cite{BridgelandCon} to the Riemann-Hilbert problem of Section \ref{RHprob}. For $\lambda$ away from the BPS rays we define 
\begin{equation}
    \Phi_{\beta^{\vee}}(v,w,\lambda):=\Phi_{l,\beta^{\vee}}(v,w,\lambda) 
\end{equation}
for $l$ a ray in the sector defined by two consecutive BPS rays containing $\lambda$.

\end{dfn}

\begin{conj}\label{Theorem2} Fix $t \in \mathbb{C}^{\times}$ with $\text{Im}(t)> 0$. Let $\mathcal{X}_{\beta^{\vee}}(t,\lambda)$ be as before, and let $\Phi_{\beta^{\vee}}(v,w,\lambda)$ be as in Definition \ref{Bsol}. Then

\begin{equation}
    \mathcal{X}_{\beta^{\vee}}(t,\lambda)=\exp(2\pi i Z_{\beta^{\vee}}(t)/\lambda)\Phi_{\beta^{\vee}}(t,1,\lambda)\,.
\end{equation}
In particular, for $\lambda$ on the sector between $l_0=i\mathbb{R}_{+}t$ and $l_{-1}=i\mathbb{R}_{+}(t-1)$ the following holds:

\begin{equation}\label{Xqd}
    \mathcal{X}_{\beta^{\vee}}(t,\lambda)=H(t|1,-\lambda)e^{Q_H(t|1,-\lambda)+2\pi i Z_{\beta^{\vee}}(t)/\lambda} \,,
\end{equation}
where $H(t|\omega_1,\omega_2)$ is the quantum dilogarithm from before and 
\begin{equation}
    Q_H(t|\omega_1,\omega_2):=-\frac{\omega_1}{2\pi i\omega_2}\mathrm{Li}_2(e^{2\pi it/\omega_1})-\frac{1}{2}\log(1-e^{2\pi i t/\omega_1})+\frac{\pi}{12}\frac{\omega_2}{\omega_1} \,.
\end{equation}
\end{conj}

\begin{rem}\label{endremark} We note that \eqref{Xqd} would  imply that a certain infinite product of $\exp(\mu(z))$ (which in turn are related to Gamma functions via \eqref{Binet}) equals, up to the $e^{Q_H(t|1,-\lambda)}$ factor, a function expressed in term of double gamma functions (see Section \ref{FQD}). Such statements seem quite similar to those in \cite[Section 30]{Barnes}, where the double gamma function is expressed as an infinite product of gamma functions, providing further support for the above conjecture. 

\end{rem}


\section{Some closing remarks}

In this work we have constructed a (real) hyperk\"ahler geometry $(N,g_N,\omega_1,\omega_2,\omega_3)$ associated to the BPS structure of the resolved conifold. As we have remarked in the introduction and in the body of the paper, because the charge $\delta$ enters in our story as a flavor charge, the part of the BPS spectrum associated to multiples of $\delta$ makes no  contributions to the instanton corrections of the semi-flat HK geometry. As a consequence, the conformal limit of the HK geometry  $(N,g_N,\omega_1,\omega_2,\omega_3)$ only allows us to conjecture a relation to part of the Riemann-Hilbert problem considered in \cite{BridgelandCon}. It would be interesting (but currently unknown to us), to see if there is a relation between the real HK geometry we consider, and the complex HK geometry constructed via \cite{Bridgeland:2020zjh,Alexandrov:2021wxu}, which does take into account the full BPS spectrum. \\

On the other hand, we showed that $(N,g_N,\omega_1,\omega_2,\omega_3)$ realizes an Ooguri-Vafa-like smoothing of the semi-flat HK geometry near the locus $t=0$. This kind of behavior is conjectured to hold in \cite[Section 7]{NewCHK} for a more general class of HK metrics, and we intend to investigate this issue in a more general setting in the future.\\

We note that a crucial ingredient for determining the full hyperk\"ahler geometry in the setting of GMN which we use is the explicit knowledge of the BPS spectrum or DT invariants which we have taken as a given in this work. In general the problem of determining the BPS spectrum for a given CY geometry is very hard.  In the context of compact CY threefolds, the full knowledge of the spectrum and its wall-crossing structure has been so far elusive. To bypass this problem, a connection between BPS structures and topological string theory may prove very helpful. This connection has been significantly substantiated in the context of the Riemann-Hilbert problems of Bridgeland \cite{BridgelandDT}. Based on the pleasant analytic properties of the tau-function for the resolved conifold \cite{BridgelandCon} and the fact that it contains the generating function of GW invariants in an asymptotic expansion, it was suggested that it provides a non-perturbative definition of the topological string partition function for this geometry. \\

The same analytic functions entering the construction of the tau-function were identified in \cite{alim2021integrable} as the solution of a difference equation for the Gromov-Witten potential of the resolved conifold \cite{alim2020difference}. The latter was derived based solely on the knowledge of the asymptotic expansion of topological string theory. The non-perturbative topological string content of this solution was extracted \cite{Alim:2021ukq}, matching it to previously known and expected results in the literature \cite{Hatsuda:2013oxa,Hatsuda:2015owa} and thus confirming the expectation of \cite{BridgelandCon}. We note here that an intermediate step of identifying the solution of the difference equation in \cite{alim2021integrable} is the relation of the non-perturbative completion of the topological string free energy to the quantum dilogarithm function \cite{alim2021integrable}:
\begin{equation}\label{diffeqqdilog}
F_{np}(\lambda,t+\check{\lambda})- F_{np}(\lambda,t)=  - \log H(t\, |\, \check{\lambda},1)\,, \quad \check{\lambda}=\frac{\lambda}{2\pi}\, ,
\end{equation}
where \cite{alim2021integrable}
\begin{equation}\label{resconfreedef}
    F_{np}(\lambda,t):= \log G_3(t\,|\,\check{\lambda},1)\,.
\end{equation}
and 
\begin{equation}\label{g3def}
    G_3(z\, | \, \omega_1,\omega_2) := \exp\left(\frac{\pi i}{6} \cdot B_{3,3}(z+\omega_1\,|\,\omega_1,\omega_1,\omega_2)\right) \cdot \sin_3(z+\omega_1\, |\, \omega_1,\omega_2,\omega_3).
\end{equation}

A relationship of the kind (\ref{diffeqqdilog}), relating topological string theory to Darboux coordinates was expected to hold more generally in  \cite[Sec. 9]{Coman:2018uwk} and was crucial in \cite{Coman:2020qgf} where a tau-function which is proposed to access the non-perturbative structure of topological string theory is related to the Darboux coordinates. Thus, understanding the map between topological strings and the BPS RH problem more precisely may be of further benefit for studying generalizations of the hyperk\"ahler geometry addressed in this work.

\begin{appendix}
\section{The mirror of the resolved conifold}\label{appendix:mirrorsymmetry}
To further motivate the choice of the central charges in the main body of the paper, we discuss the special geometry of the resolved conifold as  obtained from mirror symmetry. The mirror of non-compact CY threefolds is discussed in \cite{Chianglocal,Hori:2000kt}, see also \cite{Hosonolocal}. We focus here on the toric cases. The non-compact CY threefolds in these cases are given by
\begin{equation}
    X= \frac{\mathbb{C}^{3+k}\setminus S}{(\mathbb{C}^*)^k}\,,
\end{equation}
where the $k$ algebraic tori $\mathbb{C}^*$ act on the space by
  \begin{equation}
\begin{split}
  (\mathbb{C}^*)^a :  (z_1,\dots,z_j,\dots z_{3+k}) \mapsto  (\lambda^{l^{(a)}_1}\,z_1,\dots\lambda^{l^{(a)}_j}\,z_j,\dots,\lambda^{l^{(a)}_{3+k}}\,z_{3+k})\,, \quad a=1,\dots,k.
  \end{split}
\end{equation}
 Here, $\lambda \in \mathbb{C}^*$, $l^{(a)}_{i} \in \mathbb{Z}$ are the toric charges and $S$ is a subset  which is fixed by a subgroup of $(\mathbb{C}^*)^k$. The resolved conifold geometry corresponds to the toric variety associated to the toric charge vector
\begin{equation}
\begin{array}{ccccc}
l=&(1&1&-1&-1)
\end{array}\, .
\end{equation}

To specify the mirror geometry we first consider the variables $y_i \in \mathbb{C}^*,i=0,\dots,3$, subject to the constraint
\begin{equation}
    \prod_{i=0}^{3} y_i^{l_i}= \frac{y_0 y_1}{y_2 y_3}=1\,,
\end{equation}
and the polynomial
\begin{equation}
    P(a,y)=\sum_{i=0}^3 a_i y_i\,, \quad a_i \in \mathbb{C}\,, i=0,\dots,3\,,
\end{equation}
which enters the definition of the Landau-Ginzburg potential of the mirror. This is given by
\begin{equation}
    W= U^2+ V^2 + P(a,y)\,,
\end{equation}
where the additional $U,V \in \mathbb{C}$ variables are an artifact of local mirror symmetry, see e.~g.~\cite{Hori:2000kt}. There is a freedom to rescale $W$ by a non-zero complex number, which we can use to set one of the $y$ variables to 1, w.l.o.g we set $y_0=1$. The mirror non-compact CY threefold of the resolved conifold is then given by
\begin{equation}
    \check{X}= \left\{ (U,V,y_1,y_2) \in \mathbb{C}^2\times (\mathbb{C}^*)^2 |W=U^2+V^2 + a_0 + a_1 y_1 + a_2 y_2 + a_3  y_1 y_2^{-1}=0 \right\}\,.
\end{equation}
The $a_i, i=0,\dots,3$ are complex parameters which determine the complex structure of $\check{X}$. The rescaling of $W$ and $y_1,y_2$ can be further used to show that the complex structure of $\check{X}$ only depends on the combination
$$z:=\frac{a_0 a_1}{a_2 a_3} \,.$$
Keeping the explicit dependence on the $a_i$ is however more convenient for the derivation of the Picard-Fuchs equations from a GKZ \cite{GKZ}
system of differential equations annihilating periods of the unique holomorphic $(3,0)$ form on $\check{X}$. The latter is given by
\begin{equation}
    \Omega = \textrm{Res}_{W=0}\, \frac{1}{W} dU\, dV\, \frac{dy_1}{y_1}\, \frac{dy_2}{y_2}\,.
\end{equation}
The periods of $\Omega$ are annihilated by the GKZ operator
\begin{equation}
    \frac{\partial}{\partial a_0}  \frac{\partial}{\partial a_1}- \frac{\partial}{\partial a_2} \frac{\partial}{\partial a_3}\,,
\end{equation}
which translates into the Picard-Fuchs operator expressed in $z$, namely
\begin{equation}
    L= (1-z)\theta^2\,, \quad \theta=z\frac{d}{dz}\,.
\end{equation}
This operator has the following solutions:
\begin{equation}
    \varpi^0=1\,, \quad \varpi^1=\frac{1}{2\pi i}\log z\,.
\end{equation}
These correspond to periods of $\Omega$ over appropriately defined compact three-cycles in $H_3(\check{X},\mathbb{Z})$. The mirror map is identified as
\begin{equation}
    t= \frac{1}{2\pi i}\log z\,.
\end{equation}
A familiar phenomenon of mirror symmetry for local CY is that the Picard-Fuchs system of the mirror does not have enough solutions to recover the expected ingredients of an affine special K\"ahler geometry. Generically it is missing expressions for periods of non-compact three-cycles. One way to recover these is to carefully define non-compact three-cycles on the geometry as was done in \cite{Hosonolocal}. Alternatively one may extend the PF operators, guided by the expectation of its general form in compact CY, when it is formulated in terms of the distinguished coordinates corresponding to the mirror map. This was done in \cite{Forbes}, which we will outline here. The guiding principle is the expected form of the PF operator in terms of the special (flat) coordinate $t$ in the case when the moduli space is complex one-dimensional. It is given by, see e.~g.~\cite{Ceresole:1992su,coxkatz}
\begin{equation}
    L = \partial_t^2 \, C_{ttt}^{-1} \, \partial_t^2
\end{equation}
where $C_{ttt}:=\frac{\partial^3}{\partial t^3} F_0$. This leads to the extended PF operator in $z$ coordinate \cite{Forbes}:
\begin{equation}
    L= \theta^2 (1-z) \theta^2\,.
\end{equation}
This operator has the following solutions:
\begin{eqnarray}\label{periods}
\varpi^0 &=& 1 \, ,\nonumber\\
\varpi^1 &=&\frac{1}{2\pi i} \log z \, ,\nonumber\\
\varpi^2 &=&\frac{1}{(2\pi i)^2} \left( \frac{1}{2} (\log z)^2 + \textrm{Li}_2(z) \right) \, ,\nonumber\\
\varpi^3 &=&\frac{1}{(2\pi i)^3}\left(- \frac{1}{6} (\log z)^3 -\log z \textrm{Li}_2(z) +2  \textrm{Li}_3(z)\right)\,.
\end{eqnarray}
We can identify the additional solutions with
\begin{equation}
    \varpi^2=: F_t\,,\quad \varpi^3= 2F_0 - t F_t\,,
\end{equation}
where $F_t:=\partial_t F_0$ and where the prepotential $F_0$ reads
\begin{equation}\label{prepot} F_0 =\frac{1}{(2\pi i)^3}\left( \frac{1}{3!} (\log z)^3+ \textrm{Li}_3(z)\right)\end{equation}
matching the expected generating function of the genus zero, degree non-zero GW invariants of the resolved conifold. We note that this basis of solutions does not have integral monodromy around $z=0$, but rather transforms as
\begin{equation}
\varpi \rightarrow \left(
\begin{array}{cccc}
1&0&0&0\\
1&1&0&0\\
1/2&1&1&0\\
-1/6&-1/2&-1&1\\
\end{array}\right) \, \cdot \varpi\,.
\end{equation}
A different choice of the basis of solutions however does lead to an integral monodromy.

\end{appendix}

\subsection*{Acknowledgements:} The authors would like to give special thanks to J\"org Teschner for many helpful discussions and suggestions related to this work. 
We have also benefited from discussions with Vicente Cort\'es, Timo Weigand, and Alexander Westphal around common research projects within the Cluster of Excellence ``Quantum Universe". While work on this project was being finished, we learned from Sergei Alexandrov and Boris Pioline that they were also working on a closely related project \cite{ConfTBAResCon}. We would like to thank them for agreeing on the coordination of the arXiv submission. The work of I.T. is funded by the Deutsche Forschungsgemeinschaft (DFG, German Research Foundation) under Germany’s Excellence Strategy – EXC 2121 Quantum Universe – 390833306.
The work of M.A and A.S. is supported through the DFG Emmy Noether grant AL 1407/2-1.\\

\bibliography{References.bib}
\bibliographystyle{alpha}

\end{document}